\documentclass[11pt, reqno]{amsart}
\usepackage{amssymb}
\usepackage{bbm}
\usepackage{amsthm}
\usepackage{nccmath}
\usepackage[bookmarks=false]{hyperref}
\usepackage{etoolbox}
\usepackage{array}
\usepackage{xparse}
\usepackage{enumitem}
\usepackage{tikz}
\usepackage{mathtools}
\usepackage[font=scriptsize]{caption}

\calclayout

\newtheorem{thm}{Theorem}[section]

\newtheorem{lemma}[thm]{Lemma}
\newtheorem{prop}[thm]{Proposition}

\theoremstyle{definition}
\newtheorem{definition}[thm]{Definition}

\newtheorem{assumption}[thm]{Assumption}

\theoremstyle{remark}

\newcommand{\thmref}[1]{Theorem~\ref{#1}}

\newcommand{\secref}[1]{\S\ref{#1}}
\newcommand{\lemref}[1]{Lemma~\ref{#1}} 
\newcommand{\propref}[1]{Proposition~\ref{#1}}

\newcommand{\assref}[1]{Assumption~\ref{#1}}

\newcommand*{\dis}{\displaystyle}

\newcommand*{\qq}{\qquad}

\newcommand*{\tx}[1]{\text{#1}}

\newcommand*{\ep}{\epsilon}

\newcommand*{\suchthat}{\, \middle| \,}

\newcommand*{\myoverline}[3]{\mkern -#1mu\overline{\mkern#1mu#3\mkern#2mu}\mkern -#2mu}	

\newcommand*{\zbar}{\myoverline{-2}{0}{\z}}

\newcommand*{\Omegabar}{\myoverline{0}{0}{\Omega} }
\newcommand*{\Ubar}{\myoverline{0}{0}{\U}}

\newcommand*{\Psizbar}{\myoverline{0}{0}{\Psi_z}}

\newcommand*{\Phibar}{\myoverline{0}{0}{\Phi}}
\newcommand*{\Phizbar}{\myoverline{0}{0}{\Phi_z}}

\newcommand*{\sbar}{\myoverline{0}{0}{s}}

\newcommand*{\ybar}{\myoverline{0}{0}{y}}
\newcommand*{\fbar}{\myoverline{0}{0}{f}}

\newcommand*{\Ybar}{\myoverline{0}{0}{Y}}

\newcommand*{\Dspbar}{\myoverline{0}{0}{\Dsp}}

\newcommand*{\half}{\frac{1}{2}}

\newcommand*{\Rsp}{\mathbb{R}}
\newcommand*{\Csp}{\mathbb{C}}
\newcommand*{\Nsp}{\mathbb{N}}
\newcommand*{\Zsp}{\mathbb{Z}}
\newcommand*{\Dsp}{\mathbb{D}}

\newcommand*{\Sone}{\mathbb{S}^1}

\newcommand*{\Lone}{L^1}
\newcommand*{\Ltwo}{L^2}

\newcommand*{\Linfty}{L^{\infty}}
\newcommand*{\Hhalf}{\dot{H}^\half}

\newcommand*{\ap}{{\alpha}}

\newcommand*{\bp}{{\beta}}

\DeclareMathOperator*{\supp}{supp}

\newcommand*{\diff}{\mathop{}\! d}

\newcommand*{\compose}[1]{\circ{#1}}

\newcommand*{\conv}{*}
\newcommand*{\Hil}{\mathbb{H}}

\newcommand*{\Imag}{\tx{Im}}
\newcommand*{\Real}{\tx{Re}}

\newcommand*{\Av}{Av}
\newcommand*{\Avg}{\Av}

\newcommand*{\sgn}{sgn}

\newcommand*{\grad}{\nabla}

\newcommand*{\pt}{\partial_t}
\newcommand*{\ps}{\partial_s}

\newcommand*{\pap}{\partial_\ap}
\newcommand*{\papabs}{\abs{\pap}}
\newcommand*{\psabs}{\abs{\ps}}

\newcommand*{\Gcal}{\mathcal{G}}

\newcommand*{\btil}{\widetilde{b}}

\newcommand*{\h}{h}

\newcommand*{\U}{U}

\newcommand*{\util}{\widetilde{u}}

\newcommand*{\w}{\omega}
\newcommand*{\wtil}{\widetilde{\omega}}

\newcommand*{\F}{F}

\newcommand*{\Phiz}{\Phi_z}

\newcommand*{\z}{z}

\newcommand*{\zone}{z_1}
\newcommand*{\ztwo}{z_2}

\newcommand*{\Z}{Z}

\newcommand*{\Zap}{\Z_{,\ap}}

\newcommand*{\nobrac}[1]{ #1 }

\DeclarePairedDelimiter{\oldbrac}{\lparen}{\rparen}			
\NewDocumentCommand{\brac}{ s o m }{						
	\IfBooleanT{#1}{
  		\IfValueT{#2}{\oldbrac[#2]{#3}}
		\IfValueF{#2}{\oldbrac{#3}} 
	}
	\IfBooleanF{#1}{
  		\IfValueT{#2}{\PackageError{mypackage}{Incorrect use of brac. Insert star}{}}
		\IfValueF{#2}{\oldbrac*{#3}} 
	}		
}

\DeclarePairedDelimiter\oldcbrac{\lbrace}{\rbrace}				
\NewDocumentCommand{\cbrac}{ s o m }{					
	\IfBooleanT{#1}{
  		\IfValueT{#2}{\oldcbrac[#2]{#3}}
		\IfValueF{#2}{\oldcbrac{#3}} 
	}
	\IfBooleanF{#1}{
  		\IfValueT{#2}{\PackageError{mypackage}{Incorrect use of cbrac. Insert star}{}}
		\IfValueF{#2}{\oldcbrac*{#3}} 
	}		
}

\DeclarePairedDelimiter\oldsqbrac{\lbrack}{\rbrack}				
\NewDocumentCommand{\sqbrac}{ s o m }{					
	\IfBooleanT{#1}{
  		\IfValueT{#2}{\oldsqbrac[#2]{#3}}
		\IfValueF{#2}{\oldsqbrac{#3}} 
	}
	\IfBooleanF{#1}{
  		\IfValueT{#2}{\PackageError{mypackage}{Incorrect use of sqbrac. Insert star}{}}
		\IfValueF{#2}{\oldsqbrac*{#3}} 
	}		
}

\DeclarePairedDelimiter{\oldabs}{\lvert}{\rvert}
\NewDocumentCommand{\abs}{ s o m }{						
	\IfBooleanT{#1}{
  		\IfValueT{#2}{\oldabs[#2]{#3}}
		\IfValueF{#2}{\oldabs{#3}} 
	}
	\IfBooleanF{#1}{
  		\IfValueT{#2}{\PackageError{mypackage}{Incorrect use of abs. Insert star}{}}
		\IfValueF{#2}{\oldabs*{#3}} 
	}		
}

\DeclarePairedDelimiterX{\oldnorm}[1]{\lVert}{\rVert}{#1}
\NewDocumentCommand{\norm}{ s o o m }{					
	\IfValueT{#2} {
		\IfBooleanT{#1}{
  			\IfValueT{#3}{\oldnorm[#2]{#4}_{#3}}
			\IfValueF{#3}{\oldnorm{#4}_{#2}} 
		}
		\IfBooleanF{#1}{
  			\IfValueT{#3}{\PackageError{mypackage}{Incorrect use of norm. Insert star}{}}
			\IfValueF{#3}{\oldnorm*{#4}_{#2}} 
		}
	}
	\IfValueF{#2} {
		\IfBooleanT{#1}{\oldnorm{#4}}	
		\IfBooleanF{#1}{\oldnorm*{#4}}		
	}	
}

\makeatletter
\def\black@#1{%
    \noalign{%
        \ifdim#1>\displaywidth
            \dimen@\prevdepth
            \nointerlineskip
            \vskip-\ht\strutbox@
            \vskip-\dp\strutbox@
            \vbox{\noindent\hbox to \displaywidth{\hbox to#1{\strut@\hfill}}}%
            \prevdepth\dimen@
        \fi
    }%
}
\makeatother

\makeatletter
\renewcommand{\tocsection}[3]{%
  \indentlabel{\@ifnotempty{#2}{\bfseries\ignorespaces#1 #2\quad}}\bfseries#3}
\renewcommand{\tocsubsection}[3]{%
  \indentlabel{\@ifnotempty{#2}{\ignorespaces#1 #2\quad}}#3}

\newcommand\@dotsep{4.5}
\def\@tocline#1#2#3#4#5#6#7{\relax
  \ifnum #1>\c@tocdepth 
  \else
    \par \addpenalty\@secpenalty\addvspace{#2}%
    \begingroup \hyphenpenalty\@M
    \@ifempty{#4}{%
      \@tempdima\csname r@tocindent\number#1\endcsname\relax
    }{%
      \@tempdima#4\relax
    }%
    \parindent\z@ \leftskip#3\relax \advance\leftskip\@tempdima\relax
    \rightskip\@pnumwidth plus1em \parfillskip-\@pnumwidth
    #5\leavevmode\hskip-\@tempdima{#6}\nobreak
    \leaders\hbox{$\m@th\mkern \@dotsep mu\hbox{.}\mkern \@dotsep mu$}\hfill
    \nobreak
    \hbox to\@pnumwidth{\@tocpagenum{\ifnum#1=1\bfseries\fi#7}}\par
    \nobreak
    \endgroup
  \fi}
\AtBeginDocument{%
\expandafter\renewcommand\csname r@tocindent0\endcsname{0pt}
}
\def\l@subsection{\@tocline{2}{0pt}{2.5pc}{5pc}{}}
\makeatother

\makeatletter
 \def\@testdef #1#2#3{%
   \def\reserved@a{#3}\expandafter \ifx \csname #1@#2\endcsname
  \reserved@a  \else
 \typeout{^^Jlabel #2 changed:^^J%
 \meaning\reserved@a^^J%
 \expandafter\meaning\csname #1@#2\endcsname^^J}%
 \@tempswatrue \fi}
\makeatother

\makeatletter
\newcommand*{\rom}[1]{\expandafter\@slowromancap\romannumeral #1@}
\makeatother

\makeatletter
\patchcmd{\@sect}{\@addpunct.}{}{}{}
\patchcmd{\subsection}{-.5em}{1em}{}{}
\makeatother

\AtBeginDocument{%
   \def\MR#1{}
}

\begin{document}

\title[2D Euler uniqueness]{Uniqueness of the 2D Euler equation on rough domains}
\author[Siddhant Agrawal]{Siddhant Agrawal$^1$}
\address{$^1$ 
Instituto de Ciencias Matem\'aticas, ICMAT, Madrid, Spain}
\email{siddhant.govardhan@icmat.es}

\author[Andrea R. Nahmod]{Andrea R. Nahmod$^2$}
\address{$^2$ 
Department of Mathematics,  University of Massachusetts,  Amherst MA 01003}
\email{nahmod@math.umass.edu}


\begin{abstract}
We consider the 2D incompressible Euler equation on a bounded simply connected domain $\Omega$. We give sufficient conditions on the domain $\Omega$ so that for all initial vorticity $\omega_0 \in L^{\infty}(\Omega)$ the weak solutions are unique. Our sufficient condition is slightly more general than the condition that $\Omega$ is a $C^{1,\alpha}$ domain for some $\alpha>0$, with its boundary belonging to $H^{3/2}(\mathbb{S}^1)$. As a corollary we prove uniqueness for $C^{1,\alpha}$ domains for $\alpha >1/2$ and for convex domains which are also $C^{1,\alpha}$ domains for some $\alpha >0$. Previously uniqueness for general initial vorticity in  $L^{\infty}(\Omega)$ was only known for $C^{1,1}$ domains with possibly a finite number of acute angled corners. The fundamental barrier to proving uniqueness below the $C^{1,1}$ regularity is the fact that for less regular domains, the velocity near the boundary is no longer log-Lipschitz. We overcome this barrier by defining a new change of variable which we then use to define a novel energy functional. 
\end{abstract}


\maketitle
\tableofcontents
\section{Introduction}

Let $\Omega \subset \Rsp^2$ be a bounded simply connected domain whose boundary is a Jordan curve. The 2D incompressible Euler equation on $\Omega$ is given by
\begin{align}\label{eq:Euler}
\begin{split}
u_t + (u\cdot\grad)u = -\grad P \quad \tx{ in } \Omega \\
\grad\cdot u = 0 \quad \tx{ in } \Omega \\
u\cdot n = 0 \quad \tx{ on } \partial \Omega
\end{split}
\end{align}
Here $u$ is the velocity, $P$ is the pressure and $n$ is the outward unit normal. The vorticity is $\omega = \grad \times u = \partial_{x_1}u_2 - \partial_{x_2}u_1$ and satisfies the transport equation
\begin{align}\label{eq:vorticity}
\omega_t + u\cdot\grad \omega = 0 \quad \tx{ in } \Omega
\end{align}
One can recover the velocity from the vorticity by the Biot-Savart law $u = \grad^\perp \Delta^{-1} \w$ where $\Delta$ is the Dirichlet Laplacian and $\grad^\perp  =  (-\partial_{x_2}, \partial_{x_1})$. The 2D Euler equation has several conserved quantities, chief among them are $\norm[L^p(\Omega)]{\w(\cdot,t)}$ for any $1\leq p\leq \infty$. The fact that $\norm[L^p(\Omega)]{\w(\cdot,t)}$ is conserved follows from the measure preserving nature of the flow map induced by the divergence free vector field $u$ and is used in an essential way to prove global existence results. In this paper the measure preserving nature of the flow map plays an important role in the proof of our uniqueness results.

The study of the well-posedness problem for the 2D Euler equation has a long history. There are two important considerations to keep in mind while talking about the well-posedness problem: one is the regularity of the initial vorticity and the other is the regularity of the boundary. Let us first consider the case of both the vorticity and boundary being regular enough. Global well-posedness of strong solutions in smooth domains was proved by Wolibner \cite{Wo33} and H\"older \cite{Ho33} (see also \cite{Mc67,Ki83}). One of the most important works in the well-posedness theory is the work of Yudovich \cite{Yu63} who established global well-posedness for weak solutions on smooth domains for initial data $\w_0 \in \Lone(\Omega)\cap\Linfty(\Omega)$ (see also \cite{Ba72,Te75}). The uniqueness result of Yudovich used the Eulerian formulation and relied on the Calder\'on Zygmund inequalities,
\begin{align}\label{eq:CZineq}
\norm[L^p(\Omega)]{\grad u(\cdot,t)} \leq Cp\norm[L^p(\Omega)]{\w(\cdot,t)} \qq \tx{ for all } p\in [2,\infty)
\end{align}
Later on Marchioro and Pulvirenti \cite{MaPu94} gave a different proof of uniqueness by using the Lagrangian formulation which relied on the log-Lipschitz nature of the velocity
\begin{align}\label{eq:logLipest}
\sup_{x,y \in \Omega} \frac{\abs{u(x,t) - u(y,t)}}{\abs{x-y}\max\cbrac{-\ln\abs{x-y}, 1}} \leq C\norm[\Lone(\Omega)\cap\Linfty(\Omega)]{\w(\cdot,t)}
\end{align}
These estimates hold for $C^{1,1}$ domains but may not hold for less regular domains. For example if the domain is $C^{1 , \alpha}$ for some $0< \alpha < 1$, then in general $\grad u \notin L^p(\Omega)$ for all $p > \frac{2}{1-\alpha}$ (see \lemref{lem:gradunotinLp}). 

For the case of initial vorticity being less regular, global existence of weak solutions was proved by DiPerna and Majda \cite{DiMa87} for $\w_0 \in \Lone(\Rsp^2)\cap L^p(\Rsp^2)$ for $p>1$. 
For $p<\infty$, uniqueness is not expected in general and this is a major open problem (see the works \cite{Vi18a,Vi18b, BrSh21,BrMu20}).

For the case of boundary being less regular, global existence of weak solutions for bounded convex domains was proved by Taylor \cite{Ta00} and for arbitrary simply connected bounded domains (and exterior domains) was proved by Gerard-Varet and Lacave \cite{GeLa13,GeLa15}. These results establish global existence of weak solutions for initial vorticity $\w_0 \in \Lone(\Omega)\cap L^p(\Omega)$ for $1 < p \leq \infty$. However even for $\w_0 \in \Lone(\Omega)\cap \Linfty(\Omega)$ the question of uniqueness is a major open problem. It is important to note that if the domain is less regular, then even if the initial vorticity is assumed to be smooth, the uniqueness question does not become simpler as the regularity of the vorticity can be destroyed at a later time (see \cite{KiZl15, AlCrMa19}).

There have been some recent works that establish uniqueness for rough domains with initial vorticity $\w_0 \in \Lone(\Omega)\cap\Linfty(\Omega)$. One strategy used was to identify domains rougher than $C^{1,1}$ which satisfy either \eqref{eq:CZineq} or \eqref{eq:logLipest} and use this to prove uniqueness. This was first achieved by Bardos, Di Plinio and Temam \cite{BaDiTe13} for rectangular domains and for $C^2$ domains which allow corners of angle $\pi/m$ for $m\in \Nsp, m\geq 2$. Later Lacave, Miot and Wang \cite{LaMiWa14} proved uniqueness for $C^{2,\alpha}$ domains with a finite number of acute angled corners, and then Di Plinio and Temam \cite{DiTe15} proved uniqueness for $C^{1,1}$ domains with finitely many acute angled corners. Note that for angles bigger than $\pi/2$, the estimates \eqref{eq:CZineq} and \eqref{eq:logLipest} fail to hold and uniqueness is open in general. 

Another strategy used to prove uniqueness is to prove it for initial vorticity which is constant around the boundary. The idea behind this strategy is that if the vorticity is constant around the boundary, then the uniqueness proof of Marchioro and Pulvirenti \cite{MaPu94} works, as in this case one only needs the estimate \eqref{eq:logLipest} for $x,y \in K$ where $K\subset \Omega$ is a compact set outside of which the vorticity is constant. The strategy thus reduces to showing that if the vorticity is initially constant around the boundary, then it remains constant for later times. Lacave \cite{La15} proved that if the domain is $C^{1,1}$ with finitely many corners with angles greater than $\pi/2$ and $\w_0$ is constant around the boundary and has a definite sign, then $\w$ remains constant around the boundary for all time and the weak solutions are unique.  Lacave and Zlato\v{s} \cite{LaZl19}  proved the same result removing the restriction of definite sign on $\w_0$ but keeping the initial vorticity constant around the boundary and the corners are now only allowed to be in $(0,\pi)$.  Recently Han and Zlato\v{s} \cite{HaZl21, HaZl22} generalized the results of \cite{La15,LaZl19}, by proving uniqueness in more general domains which include $C^{1,\alpha}$ domains and convex domains, but for initial vorticity which is still constant around the boundary.

In \cite{AgNa22} the authors considered a domain with an obtuse angled corner and proved uniqueness under the assumption that the initial vorticity is non-negative and is supported on one side the angle bisector. The main idea behind this result is that the support of the vorticity moves away from the corner for $t>0$ and one can then use a time dependent weight to close the energy estimate. However it is not clear how to remove the restrictions on the support or the sign of the vorticity by this method. 

As the above discussion shows, the fact that the velocity is not log-Lipschitz near the boundary is still a fundamental obstacle to prove uniqueness for domain less regular than $C^{1,1}$ domains and there has been no way till now to overcome this obstacle without making some very stringent assumptions on the initial vorticity. 

In this paper we decisively overcome this barrier and prove uniqueness for general initial vorticity in $\Linfty(\Omega)$ on a large class of domains that are less regular than $C^{1,1}$ domains. We now give sufficient conditions on the domain for uniqueness to hold for all initial vorticity in $\Linfty(\Omega)$. The conditions we assume on the domain are as follows: 

\begin{assumption}\label{ass:main}
Let $\Omega \subset \Rsp^2$ be a bounded simply connected domain whose boundary is a Jordan curve and let $\Psi:\Dsp \to \Omega$ be a Riemann map. We assume the following properties on $\Psi$:
\begin{enumerate}
\item There exists constants $c_0, c_1 >0$ such that for all $z \in \Dsp$ we have $c_0 \leq \abs{\Psi_z(z)} \leq c_1$. 
\item Let $\Gcal: [0,1) \times \Sone \to \Csp$ be given by $\Gcal(r, e^{i\theta}) = \Psi_z(re^{i\theta})$. There exists $c_2> 0$ so that 
\begin{align*}
\sup_{r \in [0,1)} \norm[\Hhalf(\Sone)]{\Gcal(r, \cdot)} \leq c_2
\end{align*}
\item Let $\Hil$ be the Hilbert transform on $\Sone$ (see \eqref{def:Hil}). There exists $c_3>0$ so that
\begin{align*}
\sup_{r \in [0,1)} \norm[\Linfty(\Sone)]{\Hil(\abs{\Gcal}^2(r, \cdot))} \leq c_3
\end{align*}
\end{enumerate}
\end{assumption}
The first condition says that the Riemann map is bi-Lipschitz. The second condition says that the boundary of the domain in conformal coordinates is in $H^{3/2}(\Sone)$. The third condition is a technical condition which essentially says that the boundary of the domain in conformal coordinates is slightly better than Lipschitz. \assref{ass:main} is satisfied by a large class of domains as \propref{prop:main} below shows. We can now state our main results. 

\begin{thm}\label{thm:main}
Consider the Euler equation \eqref{eq:Euler} in a domain $\Omega \subset \Rsp^2$ satisfying \assref{ass:main}. Then for any initial vorticity $\w_0 \in \Linfty(\Omega)$ there exists a unique Yudovich weak solution in the time interval $[0,\infty)$ with this initial data. 
\end{thm}

The definition of Yudovich weak solutions is given in section \secref{sec:weak}. 

\begin{prop}\label{prop:main}
Let $\Omega \subset \Rsp^2$ be a bounded simply connected domain whose boundary is a Jordan curve. In addition assume that $\Omega$ satisfies any one of the following conditions:
\begin{enumerate}
\item $\Omega$ is a $C^{1,\alpha}$ domain for some $\alpha >0$ whose boundary admits a $C^1$ parametrization $f: \Sone \to \partial \Omega$  satisfying $\abs{f'} \neq 0$ on $\Sone$ and $f \in H^{3/2}(\Sone)$. 
\item $\Omega$ is a $C^{1,\alpha}$ domain for some $\alpha  > \half$.
\item $\Omega$ is a $C^{1,\alpha}$ domain for some $\alpha >0$ and $\Omega$ is convex. 
\end{enumerate}
Then $\Omega$ satisfies \assref{ass:main} and therefore Yudovich weak solutions are unique. 
\end{prop}

\thmref{thm:main} is the first result to prove uniqueness for general initial vorticity in $\Linfty(\Omega)$, where the domain is such that the velocity is not log-Lipschitz near the boundary. For the class of domains considered by \assref{ass:main} which are not $C^{1,1}$ domains, previously uniqueness was only known for initial vorticity which is constant around the boundary (see \cite{HaZl21}). The main contribution of this paper is to remove the assumption that the initial vorticity be constant around the boundary for the class of domains satisfying \assref{ass:main}.

Let us now make a few remarks about \thmref{thm:main} and \propref{prop:main} above. First, the results proved here are for bounded domains but the same result with essentially the same proof works for domains which are homeomorphic to the upper half plane. Also for simplicity we have worked with simply connected domains. However one can prove similar results for domains with a finite number of obstacles, with each obstacle being homeomorphic to the disc (to do this, we can simply define the change of variable in a neighborhood of each connected component of the boundary separately and then patch them together to get a global change of variable. See the discussion regarding the proof strategy below for the nature of the change of variable).

If $\Omega$ is a $C^{1,\alpha} $ domain for some $\alpha>0$, then it will satisfy the first and third conditions of \assref{ass:main}. So for $C^{1,\alpha} $ domains one only needs to check the second condition of \assref{ass:main}, namely that the boundary in conformal coordinates is in $H^{3/2}(\Sone)$. Alternatively as stated in \propref{prop:main}, one could instead assume that there exists a $C^1$ parametrization of the boundary which is in $H^{3/2}(\Sone)$. 

Note that in terms of the scaling of the spaces involved\footnote{This is a heuristic for bounded domains as one does not actually have scaling invariance for bounded domains. This heuristic is rigorous for unbounded domains which are homeomorphic to the upper half plane.} the regularity assumptions on the domain in \assref{ass:main} are of one order lower than the assumption that the domain be a $C^{1,1}$ domain. For example in this paper the boundary is assumed to be in $H^{3/2}(\Sone)$ instead of being in $C^{1,1}(\Sone)$, which shows that the number of derivatives required is reduced by half and the integrability requirement is reduced from $\Linfty(\Sone)$ to $\Ltwo(\Sone)$. It is  important to note however that the above result does not show uniqueness for general $C^{1,\alpha}$ domains for $0<\alpha \leq \half$. This is because in general, a function belonging to $C^{1,\half}(\Sone)$ does not always belong to $H^{3/2}(\Sone)$ (see \cite{MiSi15}). Also note that \assref{ass:main} excludes domains with corners, as the first condition of \assref{ass:main} will be violated, moreover the second condition will also be violated for angles less than $\pi$. 
\medskip\smallskip

\underline{Strategy of the Proof:} First recall that the energy used in the uniqueness proof of Marchioro and Pulvirenti \cite{MaPu94} is 
\begin{align}\label{eq:EnergystandardX}
E_1(t) = \int_\Omega \abs{X_1(x,t) - X_2(x,t)} \diff x
\end{align}
where $X_1(\cdot,t), X_2(\cdot,t) : \Omega \to \Omega$ are the flow maps for the two solutions. One can use this energy to prove uniqueness for $C^{1,1}$ domains but the proof does not go through for less regular domains as the velocity is no longer log-Lipschitz (see \lemref{lem:gradunotinLp}). Consider a conformal map $\Psi:\Dsp \to \Omega$ and let $Y_i(y,t) = \Psi^{-1}(X_i(\Psi(y),t))$ for $i = 1,2$ and $y \in \Dsp$ be the corresponding flow maps in $\Dsp$. One can prove the same uniqueness result for $C^{1,1}$ domains by  using instead the energy
\begin{align}\label{eq:EnergystandardY}
E_2(t) = \int_\Dsp \abs{Y_1(y,t) - Y_2(y,t)} \abs{\Psi_z(y)}^2 \diff y
\end{align}
Here $\abs{\Psi_z(y)}^2 \diff y$ is the pullback of the Lebesgue measure on $\Omega$. Again this energy is not good enough to prove uniqueness for domains less regular than $C^{1,1}$ domains. However it is now easier to see where the main difficulty in proving uniqueness comes from. The differential equation for $Y(y,t)$ is as follows
\begin{align}\label{eq:Yytintro}
\frac{\diff Y(y,t)}{\diff t} = \btil(Y(y,t),t)\abs{\Psi_z(Y(y,t))}^{-2}
\end{align}
where $\btil$ is given by \eqref{eq:btil}. Now $\btil(\cdot,t)$ is log-Lipschitz in $\Dspbar$ from \lemref{lem:btil} part (2). However if the domain is less regular than $C^{1,1}$, then $\abs{\Psi_z}^{-2}$ is no longer log-Lipschitz in $\Dsp$, in fact it is only $C^{\alpha}$ for a $C^{1,\alpha}$ domain (if $0<\alpha < 1$).

To overcome this regularity hurdle we introduce a new energy defined by
\begin{align}\label{eq:Etdefintro}
E(t) = \int_\Dsp \abs{F(Y_1(y,t)) - F(Y_2(y,t)) } \abs{\Psi_z(y)}^2 \diff y
\end{align}
where $F:\Dsp \to \Dsp$ is a carefully constructed diffeomorphism. The map $F$ is the identity map away from the boundary and near the boundary we have $F(re^{i\theta}) = re^{iL(r,\theta)}$ where $\partial_\theta L(r,\theta) = C(r)\abs{\Psi_z(r,\theta)}^{2}$, where $C(r)$ is a normalizing constant. This choice of $F$ allows us to eliminate the factor $\abs{\Psi_z(Y(y,t))}^{-2}$ in \eqref{eq:Yytintro} at least in the $\partial_\theta$ direction when we take the time derivative to $E(t)$. To close the energy estimate and control the $\partial_r$ derivative, we then use the fact that the radial component of $\btil$ decays as $\phi(1 - r)$ as $r \to 1$ (where $\phi$ is given by \eqref{def:phi} below). For the argument to work, we need $F$ to satisfy several properties, the most important being that $F$ needs to be bi-Lipschitz on $\Dsp$. To prove that \assref{ass:main} implies that $F$ is bi-Lipschitz, we crucially use the special structure of the function  $\abs{\Psi_z(r,\theta)}^{2}$, namely that $\Psi_z$ is a holomorphic function.  \assref{ass:main}  is very close to being equivalent to the condition that $F$ is bi-Lipschitz and luckily it turns out that \assref{ass:main} is also sufficient to prove several other estimates for $F$ (such as  \lemref{lem:FrFthetaestimates}). 

For domains satisfying \assref{ass:main}, all the energies defined in \eqref{eq:EnergystandardX}, \eqref{eq:EnergystandardY} and \eqref{eq:Etdefintro} are equivalent i.e. $E_1(t) \approx_\Omega E_2(t) \approx_\Omega E(t)$. However {\it{the time derivative}} of $E(t)$ is much better behaved and this is what is crucially used to prove uniqueness.

The paper is organized as follows: In \secref{sec:notation} we introduce the notation and derive the flow equation. In \secref{sec:weak} we introduce the notion of Yudovich weak solutions and establish some of their basic properties. In \secref{sec:changevar} we define our change of variable and prove important estimates for it. Finally in \secref{sec:energy} we use the energy \eqref{eq:Etdefintro} and the estimates proved for the change of variable in \secref{sec:changevar} to prove \thmref{thm:main} and \propref{prop:main}. The appendix \secref{sec:appendix} contains some basic estimates used throughout the paper.

\medskip
\noindent \textbf{Acknowledgment}: S.A. received funding from the European Research Council (ERC) under the European Union’s Horizon 2020 research and innovation program through the grant agreement 862342. A.N. is funded in part by NSF DMS-2052740, NSF DMS-2101381 and the Simons Foundation Collaboration Grant on Wave Turbulence (Nahmod's award ID 651469).

\section{Notation and Preliminaries}\label{sec:notation}

We will identify $\Rsp^2 \simeq \Csp$ and we denote a ball of radius $r$ by $B_r(z_0) = B(z_0, r) = \cbrac{z\in \Csp \suchthat \abs{z-\z_0} <r}$. Let $\Dsp = B_1(0)$ be the unit disc and let $\Sone = \partial \Dsp$. For $z \in \Csp$ we define $z^*  = z \abs{z}^{-2}$ (here $0^* = \infty$). We define the function $\phi:[0,\infty) \to \Rsp$ as $\phi(0) =0 $ and for $x>0$ as  
\begin{align}\label{def:phi}
\phi(x) = x\max\cbrac{-\ln(x), 1}
\end{align}
Given the local log-Lipschitz nature of the velocity, this function $\phi$ will play an important role in closing the energy estimate in \secref{sec:energy}. Observe that $\phi$ is a continuous increasing function on $[0,\infty)$ with  $x\leq \phi(x) $ for all $x\geq 0$ and that $\phi$ is a concave function on $[0,1/10]$. Also observe that if $c\geq 1$ then $\phi(cx) \leq c\phi(x)$ for all $x\geq 0$. 

We write $a \lesssim b$ if there exists a universal constant $C>0$ so that $a\leq Cb$. We write $a\lesssim_{\eta} b$ if there exists a constant $C = C(\eta)>0$ depending only on $\eta$ so that $a\leq Cb$. Similar definitions for $\lesssim_{\eta_1,\eta_2}$, $\lesssim_{\eta_1,\eta_2, \eta_3}$ etc. We write $a \approx b$ if $a \lesssim b$ and $b\lesssim a$. Similarly we write $a \approx_\eta b$ if $a\lesssim_\eta b$ and $b \lesssim_\eta a$ etc. We will write $a \lesssim_\Omega b$ if the constants depend on the constants in \assref{ass:main}.

To simplify notation, in the following we will identify functions $f: \Sone \to \Csp$ with their pullbacks $\tilde{f}: \Rsp \to \Csp$, where $\tilde{f}(\ap) = f(e^{i\ap})$. Similarly we will identify functions $F: \Dspbar \to \Csp$ with their pullbacks $\tilde{F}: [0,1] \times \Rsp \to \Csp$ where $\tilde{F}(r,\theta) = F(re^{i\theta})$. If $g:\Rsp \to \Csp$ is a $2\pi$ periodic function, then in this paper whenever we use the $L^p$ or Sobolev norms of $g$, what we mean is that we are computing the norms by looking at $g$ as a function on $\Sone$ and not as a function on $\Rsp$. 

We define the Fourier transform for a function $f: \Sone \to \Csp$ as
\begin{align*}
\hat{f}(n) & = \frac{1}{2\pi}\int_0^{2\pi} f(\ap) e^{-in\ap} \diff \ap 
\end{align*}
and the inverse Fourier transform as
\begin{align*}
 f(\ap) & = \sum_{n = -\infty}^{\infty} \hat{f}(n) e^{in\ap}
\end{align*}
The $L^p$ norm of $f$ is defined as 
\begin{align*}
\norm[p]{f} = \brac{\int_0^{2\pi} \abs*{f(\ap)}^p  \diff \ap}^{\frac{1}{p}}
\end{align*}
A Fourier multiplier with symbol $a(\xi)$ is the operator $T_a$ defined formally by the relation $\dis \widehat{T_a{f}} = a(\xi)\hat{f}(\xi)$. For $s \geq 0$ the operators $\papabs^s $ and $\langle\pap\rangle^s$ are defined as the Fourier multipliers with symbols $\abs{\xi}^s$ and $(1 + \abs{\xi}^2)^{\frac{s}{2}}$ respectively. 
The Sobolev space $H^s(\Sone)$ for $s\geq 0$  is the space of functions with  $\norm[H^s]{f} = \norm*[\Ltwo]{\langle\pap\rangle^s f} < \infty$. The homogenous Sobolev space $\Hhalf(\Sone)$ is the space of functions modulo constants with  $\norm[\Hhalf]{f} = \norm*[\Ltwo]{\papabs^\half f} < \infty$. 
Hence in particular we observe that
\begin{align*}
\norm[\Hhalf]{f}^2 = \norm*[\big][2]{\papabs^\half f}^2 = \int_0^{2\pi} \fbar \papabs f \diff \ap
\end{align*}
The Poisson kernel for the disc is given by
\begin{align}\label{eq:Poisson}
P_{r}(\theta) = \frac{1 - r^2}{1 -2r\cos(\theta) + r^2} \qq \tx{ for } 0\leq r < 1
\end{align}
Define the averaging operator
\begin{align*}
\Av(f) & = \frac{1}{2\pi}\int_0^{2\pi} f(\bp) \diff \bp  
\end{align*}
and the Hilbert transform 
\begin{align}\label{def:Hil}
(\Hil f)(\ap) & = \frac{1}{2\pi} p.v. \int_0^{2\pi} \cot\brac{\frac{\ap - \bp}{2}}   f(\bp)\diff \bp
\end{align}
We refer the reader to \cite{MuSc13} for basic properties of the Hilbert transform. Observe that we have 
\begin{align*}
\Av(f) & = \hat{f}(0) \\
\widehat{(\Hil f)}(n)  & = -i \sgn(n) \hat{f}(n) \\
\papabs  & = \Hil \partial_{\alpha} 
\end{align*}
where
\begin{align*}
\sgn(n) = 
\begin{cases}
1 \quad \tx{ if } n \geq 1 \\
0 \quad \tx{ if } n = 0 \\
-1 \quad \tx{ if } n \leq -1
\end{cases}
\end{align*}

Let us now derive the equation of the flow. If the Green's function of the domain $\Omega$ is $G_{\Omega}(x,y)$, then the kernel of the Biot-Savart law is $ K_{\Omega}(x,y) := \grad_x^\perp G_{\Omega}(x,y)$ with $\grad^\perp_x  =  (-\partial_{x_2}, \partial_{x_1})$. Let $\Phi:\Omega \to \Dsp$ be a Riemann map and let $\Psi: \Dsp \to \Omega$ be the inverse of $\Phi$. Fix $\ztwo \in \Omega$ and let $f: \Omega\backslash\cbrac{\ztwo} \to \Csp$ be defined as  
\begin{align}\label{eq:f}
f(\z) =  \frac{1}{2\pi}\ln\brac*[\Bigg]{\frac{\Phi(z) - \Phi(\ztwo) }{(\Phi(z) - \Phi(\ztwo)^*)\Phi(\ztwo) }}
\end{align}
Clearly $f$ is holomorphic and we have for $\zone \in \Omega$, $\zone\neq \ztwo$
\begin{align}\label{def:Greensfunc}
G_{\Omega}(\zone,\ztwo) =  \frac{1}{2\pi}\ln\abs*[\Bigg]{\frac{\Phi(z) - \Phi(\ztwo) }{(\Phi(z) - \Phi(\ztwo)^*)\Phi(\ztwo) }} = \Real\cbrac{f(\zone)}
\end{align}
Hence
\begin{align*}
K_{\Omega}(\zone,\ztwo) & = \Real\begin{pmatrix} -\partial_{x_2} f(\zone) \\ \partial_{x_1} f(\zone) \end{pmatrix} \\
& = \Real(-i f_{x_1}(\zone)) + i \Real(f_{x_1}(z_1)) \\
& = i\overline{f_z}(z_1)
\end{align*}
Then from \eqref{eq:f} we have
\begin{align}\label{eq:Kernel}
K_{\Omega}(\zone,\ztwo) =  \brac{\frac{i}{2\pi}}\Phizbar(\zone)\sqbrac{\frac{1}{\Phibar(\zone) - \Phibar(\ztwo)} - \frac{1}{\Phibar(\zone) - \frac{1}{\Phi(\ztwo)}}}
\end{align}
If $\w(\cdot,t)$ is the vorticity at time $t$, then from the Biot-Savart law we see that 
\begin{align*}
u(z_1,t) = \int_{\Omega} K_{\Omega}(z_1,z_2)\w(z_2,t)\diff z_2
\end{align*}
Now the equation for the flow $X:\Omega\times [0,\infty) \to \Omega$ is given by
\begin{align}\label{eq:ODEXfundamental}
\frac{\diff X(x,t)}{\diff t} & = u(X(x,t),t) = \int_{\Omega} K_{\Omega}(X(x,t),z)\w(z,t) \diff z
\end{align}
with $X(x,0) = x$ for all $x \in \Omega$. Hence we have
\begin{align*}
\frac{\diff X(x,t)}{\diff t} = \brac*[\Big]{\frac{i}{2\pi}}\Phizbar(X(x,t)) \int_{\Omega}\sqbrac{\frac{1}{\Phibar(X(x,t)) - \Phibar(z)} - \frac{1}{\Phibar(X(x,t)) - \frac{1}{\Phi(z)}}}\w(z,t) \diff z
\end{align*}
Define the function $b:\Omega\times [0,\infty) \to \Csp$ as 
\begin{align}\label{eq:b}
b(x,t) =  \brac*[\Big]{\frac{i}{2\pi}} \int_{\Omega}\sqbrac{\frac{1}{\Phibar(x) - \Phibar(z)} - \frac{1}{\Phibar(x) - \frac{1}{\Phi(z)}}}\w(z,t) \diff z
\end{align}
Hence the equation for $X$ can be written as
\begin{align}\label{eq:X}
\frac{\diff X(x,t)}{\diff t} = b(X(x,t),t)\Phizbar(X(x,t))
\end{align}
with $X(x,0) = x$ for all $x \in \Omega$. 

We now convert the flow equation above in $\Omega$ to a flow equation on $\Dsp$. For $x\in \Omega$, let $y \in \Dsp$ be given by $y = \Phi(x)$. Consider the flow $Y:\Dsp\times[0,\infty) \to \Dsp $ given by 
\begin{align}\label{def:Y}
Y(y,t) = \Phi(X(x,t))
\end{align}
and define $\btil:\Dsp \times [0,\infty) \to \Csp$ as $\btil(y,t) = b(x,t)$. Then $\btil(Y(y,t),t) = b(X(x,t),t)$ and we have
\begin{align*}
\frac{\diff Y(y,t)}{\diff t} = \btil(Y(y,t),t)\, \abs{\Phi_z \compose \Phi^{-1} (Y(y,t))}^2
\end{align*}
along with $Y(y,0) = y$ for all $y \in \Dsp$. Now $\Psi_z = \frac{1}{\Phi_z\compose \Phi^{-1}}$ and hence
\begin{align}\label{eq:Y}
\frac{\diff Y(y,t)}{\diff t} = \btil(Y(y,t),t)\, \abs{\Psi_z(Y(y,t))}^{-2}
\end{align}
and $Y(y,0) = y$ for all $y \in \Dsp$. We can write a simple formula for $\btil$. As $y = \Phi(x)$ we see that
\begin{align*}
\btil(y,t) = b(x,t) & =  \brac*[\Big]{\frac{i}{2\pi}} \int_{\Omega}\sqbrac{\frac{1}{\Phibar(x) - \Phibar(z)} - \frac{1}{\Phibar(x) - \frac{1}{\Phi(z)}}}\w(z,t) \diff z \\
& =  \brac*[\Big]{\frac{i}{2\pi}} \int_{\Omega}\sqbrac{\frac{1}{\ybar - \Phibar(z)} - \frac{1}{\ybar - \frac{1}{\Phi(z)}}}\w(z,t) \diff z
\end{align*}
Next, we change variables by setting $s = \Phi(z)$ with $s\in\Dsp$ and observe that $\diff s = \abs{\Phiz(z)}^2\diff z$ and hence $\diff z = \abs{\Phiz\compose \Phi^{-1} (s)}^{-2} \diff s = \abs{\Psi_z(s)}^2\diff s$. Defining $\wtil:\Dsp\times[0,\infty) \to \Rsp$ as $\wtil(s,t) = \w(z,t)$ we get
\begin{align}\label{eq:btil}
\btil(y,t) =  \brac*[\Big]{\frac{i}{2\pi}} \int_{\Dsp}\sqbrac{\frac{1}{\ybar - \sbar} - \frac{1}{\ybar - \frac{1}{s}}}\wtil(s,t)  \abs{\Psi_z (s)}^{2} \diff s
\end{align}

\section{Weak Solutions}\label{sec:weak}

In this section we define Yudovich weak solutions and establish their existence and prove some of their basic properties. The existence of weak solutions will follow directly from the work \cite{GeLa13} and their properties will also follow easily. For the definition of weak solution we closely follow the definition as given in \cite{GeLa13,GeLa15}. In this paper we use the definition of Yudovich weak solution as given in \cite{LaZl19}. 

Let $\Omega$ satisfy \assref{ass:main} and let $T>0$. We say that $(u,\w)$ is in the Yudovich class in the time interval $[0,T)$ if 
\begin{align}\label{eq:Yudovich}
u \in \Linfty([0,T); \Ltwo(\Omega)), \qq  \w = \grad \times u \in \Linfty([0,T);  \Linfty(\Omega))
\end{align}
Now consider initial data $(u_0,\w_0)$ satisfying
\begin{align}\label{eq:initial}
\begin{split}
&u_0 \in \Ltwo(\Omega), \qq \w_0 = \grad \times u_0 \in \Linfty(\Omega),   \\
& \int_{\Omega} u_0 \cdot h = 0 \quad \forall h \in \mathcal{H}(\Omega) = \cbrac{ h \in \Ltwo(\Omega) \suchthat h = \grad p \tx{ for some } p \in H^1(\Omega)}
 \end{split}
\end{align}
\begin{definition}
We say that $(u,\w)$ is a Yudovich weak solution to the Euler equation \eqref{eq:Euler} with initial condition $(u_0,\w_0)$ in the time interval $[0,T)$, if $(u,\w)$ is in the Yudovich class \eqref{eq:Yudovich} and satisfies
\begin{align}\label{eq:weak}
\int_0^T \int_{\Omega} \w(\pt \varphi + u\cdot \grad \varphi) \diff x \diff t = - \int_{\Omega} \w_0 \varphi(\cdot,0) \diff x \qq \forall \varphi \in C^{\infty}_c (\Omega\times[0,T)),
\end{align}
and for $a.e.$ $t\in [0,T)$ we have  
\begin{align}\label{eq:divweak}
\int_{\Omega} u(\cdot,t)\cdot h = 0 \qq \forall h\in \mathcal{H}(\Omega).
\end{align}
\end{definition}

To prove some of the properties of weak solutions, we will need to first prove some estimates regarding the function $\btil$ defined in \eqref{eq:btil}. 
\begin{lemma}\label{lem:btil}
Let $\Omega \subset \Rsp^2$ be a bounded simply connected domain with Jordan boundary. Let $\Psi:\Dsp \to \Omega$ be a Riemann map and assume that $\norm[\Linfty(\Omega)]{\Psi_z} \leq C_1$ for some $C_1>0$. Fix $t \in [0,\infty)$ and let $\wtil(\cdot, t) \in \Linfty(\Omega)$ with $\norm[\Linfty(\Omega)]{\wtil(\cdot, t)} \leq C_2$ for some $C_2>0$. Then the function $\btil(\cdot,t) : \Dspbar \to \Csp$ defined by \eqref{eq:btil} satisfies the following properties:
\begin{enumerate}
\item $\abs*[\big]{\btil(z,t)} \lesssim_{C_1, C_2} 1$ for all $z \in \Dspbar$ 
\vspace*{1mm}
\item $\abs*[\big]{\btil(z_1,t) - \btil(z_2,t)} \lesssim_{C_1, C_2} \phi(\abs{z_1 - z_2})$ for all $z_1, z_2 \in \Dspbar$
\vspace*{1mm}
\item $\Real(\btil(z, t)\zbar) = 0$ for all $z \in \Sone$
\vspace*{1mm}
\item $\abs{\Real\cbrac{\btil(z,t) \frac{\zbar}{\abs{z}} }} \lesssim_{C_1, C_2} \phi(1 - \abs{z})$ for all $z \in \Dspbar$, $z \neq 0$
\end{enumerate}
\end{lemma}
\begin{proof}
The first estimate follows directly from the formula \eqref{eq:btil}. Now let $z_1, z_2 \in \Dspbar$. Observe that for $z, s \in \Dsp$ we have $\abs{z - \frac{1}{\sbar}} \geq \abs{z - s}$. Hence
\begin{align*}
\abs*[\big]{\btil(z_1,t) - \btil(z_2,t)} &  \lesssim_{C_1, C_2} \int_{\Dsp} \frac{\abs{z_1 - z_2}}{\abs{\zbar_1 - \sbar}\abs{\zbar_2 - \sbar}} \diff s + \int_{\Dsp} \frac{\abs{z_1 - z_2}}{\abs{\zbar_1 - \frac{1}{s}}\abs{\zbar_2 - \frac{1}{s}}} \diff s \\
& \lesssim_{C_1, C_2} \int_{\Dsp} \frac{\abs{z_1 - z_2}}{\abs{\z_1 - s}\abs{\z_2 - s}} \diff s 
\end{align*}
The second estimate now follows from \lemref{lem:phiabest}. Now for $y \in \Sone$ we have $\ybar = \frac{1}{y}$ and hence
\begin{align}\label{eq:btilboundary}
\begin{aligned}
\ybar\,\btil(y,t) & = \brac{\frac{i}{2\pi}} \ybar \int_{\Dsp}\sqbrac{\frac{1}{\ybar - \sbar} - \frac{1}{\ybar - \frac{1}{s}}}\wtil(s,t)  \abs{\Psi_z (s)}^{2} \diff s \\
& = \brac{\frac{i}{2\pi}} \ybar \int_{\Dsp}\sqbrac{\frac{\sbar - \frac{1}{s}}{\brac{\ybar - \sbar}\brac{\frac{1}{y} - \frac{1}{s}}}}\wtil(s,t)  \abs{\Psi_z (s)}^{2} \diff s \\
& =  \brac{\frac{i}{2\pi}} \int_{\Dsp} \sqbrac{ \frac{(1 - \abs{s}^2) \abs{y}^2}{\abs{y-s}^2}  }\wtil(s,t)  \abs{\Psi_z (s)}^{2} \diff s
\end{aligned}
\end{align}
Hence $\Real(\btil(y,t)\,\ybar) = 0$ for $y \in \Sone$ proving the third estimate.  Now using this we see that for $z \in \Dspbar $ and $z \neq 0$ we have
\begin{align*}
\Real\cbrac{\btil(z,t) \frac{\zbar}{\abs{z}} } = \Real\cbrac{\btil(z,t) \frac{\zbar}{\abs{z}} } - \Real\cbrac{\btil\brac{\frac{z}{\abs{z}},t} \frac{\zbar}{\abs{z}} }
\end{align*}
Hence the fourth estimate now follows from the second estimate. 
\end{proof}

Let us now state the existence result for weak solutions which follows from the very general existence result proved in \cite{GeLa13}. 

\begin{thm}\label{thm:weak}
Let $\Omega$ satisfy \assref{ass:main} and consider an initial data $(u_0,\w_0)$ satisfying \eqref{eq:initial} in the domain $\Omega$. Then there exists a Yudovich weak solution $(u,\w)$ in domain $\Omega$ in the time interval $[0,\infty)$ in the sense of \eqref{eq:weak} and \eqref{eq:divweak}. 
\end{thm}

In addition to the above existence result, we will also need some basic properties of these weak solutions. We now prove all such required properties in an all encompassing lemma, parts of which appear in similar statements in \cite{La15, LaZl19,HaZl21, AgNa22}.

\begin{lemma}\label{lem:transport}
Let $\Omega$ satisfy \assref{ass:main} and let $(u,\w)$ be a Yudovich weak solution with initial vorticity $\w_0$ in the domain $\Omega$ in the time interval $[0,\infty)$. Let $C_0 = \norm[\Linfty(\Omega\times {[}0,\infty))]{\w }$ and let $S = \cbrac{t \in {[}0,\infty) \suchthat \norm[\Linfty(\Omega)]{\w(\cdot,t)} \leq C_0}$. Clearly $S$ is a set of full measure in ${[}0,\infty)$. The weak solution satisfies the following properties:
\begin{enumerate}
\item We can redefine $u$ on a set of zero measure in $\Omega \times [0,\infty)$ so that  
\begin{align}\label{eq:ufrombiotSavart}
u(x,t) = \int_\Omega K_\Omega(x,y)\w(y,t) \diff y \qq \tx{ for all } t\in S \tx{ and } x \in \Omega
\end{align}
where $K_\Omega$ is given by \eqref{eq:Kernel}. Moreover 
\begin{align}\label{eq:ulinftybound}
\sup_{x \in \Omega}\abs{u(x,t)} \lesssim_{\Omega, C_0 } 1  \qq \tx{ for all } t \in [0,\infty)
\end{align}
and for each compact set $K \subset \Omega$ and $t \in [0,\infty)$ we have
\begin{align}\label{eq:uloglipinterior}
\sup_{x_1, x_2 \in K} \frac{\abs{u(x_1, t)  - u(x_2, t)}}{\phi(\abs{x_1 - x_2})} \lesssim_{K, \Omega, C_0 } 1
\end{align}

\item For each fixed $x \in \Omega$, the ODE defined by \eqref{eq:ODEXfundamental}, has a unique solution in all of $t \in [0,\infty)$, and $X(x,t) \in \Omega$ for all $t \in [0,\infty)$.  The map $X(\cdot,t):\Omega \to \Omega$ is a homeomorphism for each $t\in [0,\infty)$ and the functions $X,X^{-1}:\Omega\times[0,\infty) \to \Omega$ are continuous. Moreover for each $t \in [0,\infty)$, the maps $X(\cdot,t), X^{-1}(\cdot,t): \Omega \to \Omega$ are measure preserving. 
\item We have $\w(x,t) = \w_0(X^{-1}(x,t))$ for a.e. $(x,t) \in \Omega\times[0,\infty)$. We can redefine $\w$ and $u$ on a set of measure zero in $\Omega \times [0,\infty)$ so that $\w(x,t) = \w_0(X^{-1}(x,t))$ and \eqref{eq:ufrombiotSavart} hold for all $(x,t) \in \Omega \times [0,\infty)$. Also this redefinition does not change the mappings $X, X^{-1}:\Omega \times [0,\infty) \to \Omega$ and the estimates \eqref{eq:ulinftybound} and \eqref{eq:uloglipinterior} continue to hold.

\item If $(t_n)_{n=1}^\infty$ is a sequence in $[0,\infty)$ with $t_n \to t \in [0,\infty)$, then for any $1\leq p <\infty$ we have $\norm[L^p(\Omega)]{\w(\cdot,t_n) - \w(\cdot,t)} \to 0$ as $n\to \infty$. 
\item The functions $b:\Omegabar\times[0,\infty) \to \Csp$, $\btil:\Dspbar \times [0,\infty) \to \Csp$ and $u:\Omega\times[0,\infty) \to \Csp$ are bounded continuous functions.
\end{enumerate}
\end{lemma}
\begin{proof}
We prove the statements sequentially:
\begin{enumerate}
\item Let $H^1_0(\Omega)$ be the completion of $C^{\infty}_c(\Omega)$ in the $H^1(\Omega)$ norm. Let $t \in S$ and let $\psi(\cdot, t) \in H^1_0(\Omega)$ be the unique solution to
\begin{align}\label{eq:psisolvingPoisson}
\Delta \psi(\cdot, t) = \w(\cdot, t) \quad \tx{ in } \Omega, \qq\qq \psi(x,t)  = 0 \quad \tx{ for  } x \in  \partial \Omega
\end{align}
Now define $\tilde{u}:\Omega\times[0,\infty) \to \Csp$ as  $\tilde{u}(\cdot,t) = \grad^\perp \psi(\cdot,t)$ for $t \in S$ and $\tilde{u}(\cdot,t) = 0$ for $t \in [0,\infty) \backslash S$. It is easy to see that for all $t \in [0,\infty)$, $\tilde{u}(\cdot,t)$ satisfies \eqref{eq:divweak}. Now let $v:\Omega \times [0,\infty) \to \Csp$ be defined as $v = u - \tilde{u}$. Hence we see that $\grad \cdot v = \grad \times v = 0$ in $\Omega \times [0,\infty)$. Therefore there exists $S' \subset [0,\infty)$ of full measure so that for all $t \in S'$ we have $\grad \cdot v(\cdot, t) = \grad \times v(\cdot,t) = 0$. Therefore for all $t \in S'$ we see that $v(\cdot,t)$ is an anti-holomorphic function and hence $v(\cdot,t) = \grad p$ for some $p \in H^1(\Omega)$. Therefore from \eqref{eq:divweak} we see that $v(x, t) = 0$ for a.e. $(x,t) \in \Omega \times [0,\infty)$. Therefore $u = \tilde{u}$ for a.e. $(x,t) \in \Omega \times [0,\infty)$. We therefore now redefine $u$ by $\tilde{u}$. 

Now from the definition of Green's function we see that for $t \in S$, $\psi(x, t) = \int_\Omega G_\Omega(x,y) \w(y,t) \diff y$. Applying $\grad^\perp$, we obtain $u(x,t) = \int_\Omega K_\Omega(x,y)\w(y,t) \diff y$ for all $t \in S$. 

From \eqref{eq:X} we see that for $t\in S$ we have
\begin{align*}
u(x,t) = b(x,t) \Phizbar(x) = \btil(\Phi(x),t) \Phizbar(x)
\end{align*}
Now from parts (1) and (2) of \lemref{lem:btil} and \assref{ass:main} we immediately get \eqref{eq:ulinftybound} and \eqref{eq:uloglipinterior}. 

\item From \eqref{eq:ulinftybound} and \eqref{eq:uloglipinterior}, we see that the ODE \eqref{eq:ODEXfundamental} has a unique solution at least for a short time. Now let $Y(y,t)$ be given by \eqref{def:Y} and we see that it satisfies the ODE \eqref{eq:Y}. Using \lemref{lem:btil} part (4) we obtain for all $t \in S$
\begin{align*}
\abs{\frac{\diff}{\diff t} (1 - \abs{Y(y,t)})}  & =  \abs{\Real\cbrac{\frac{\Ybar(y,t)}{\abs{Y(y,t)} }\frac{\diff Y(y,t)}{\diff t}  }} \\
& = \abs{\Psi_z(Y(y,t))}^{-2} \abs{\Real\cbrac{\frac{\Ybar(y,t)}{\abs{Y(y,t)} }\, \btil(Y(y,t),t)}} \\
& \lesssim_{\Omega, C_0 } \phi(1 - \abs{Y(y,t)})
\end{align*}
Hence from \lemref{lem:phi} we see that $1 - \abs{Y(y,t)} > 0$ for all $(y,t) \in \Dsp \times [0,\infty)$. Therefore $X(x,t) \in \Omega$ for all $(x,t) \in \Omega\times [0,\infty)$ and the ODE \eqref{eq:ODEXfundamental} has a unique solution for all of $t \in [0,\infty)$. 

Now let $K \subset \Omega$ be a compact set and let $x_1, x_2 \in K$ and let $T>0$. We see that for a.e. $t \in [0,T]$ we have
\begin{align*}
\abs{\frac{\diff}{\diff t} \abs{X(x_1,t) - X(x_2,t)}} & \leq \abs{u(X(x_1,t),t) - u(X(x_2,t),t)} \\
& \lesssim_{K, \Omega, C_0 } \phi(\abs{X(x_1,t) - X(x_2, t)})
\end{align*}
Therefore from \lemref{lem:phi} we see that there exists a constant $C_1 = C_1(\Omega) > 0$ and another constant $c = c(K, \Omega, T, C_0 ) > 0$ so that for all $t \in [0,T]$ we have
\begin{align}\label{eq:Xholdercontinuity}
\sqbrac{\frac{\abs{x_1 - x_2}}{C_1}}^{e^{ct}} \leq \abs{X(x_1,t) - X(x_2, t)} \leq C_1 \abs{x_1 - x_2}^{e^{-ct}}
\end{align}
Hence $X(\cdot,t)$ is continuous and injective for all $t \in [0, T]$. To see that it is surjective, fix $(x, t) \in \Omega\times [0, T]$ and consider the function $X^*(x,t,\tau)$ for $0\leq \tau\leq t$  satisfying the ODE
\begin{align}\label{eq:Xstarxttau}
\frac{\diff X^*(x,t,\tau)}{\diff \tau} = -u(X^*(x,t,\tau), t - \tau) \quad \tx{ for } 0\leq \tau \leq t, \qq X^*(x,t,0) = x
\end{align}
Now from \eqref{eq:ulinftybound}, \eqref{eq:uloglipinterior} we see that this ODE has a unique solution and by the same argument for $X(x,t)$, we see that $X^*(x,t,\tau) \in \Omega$ for all $0\leq \tau\leq t$. From the definition of $X^*(x,t,\tau)$ we now see that $X(X^*(x,t,t),t) = x$. Hence $X(\cdot,t):\Omega \to \Omega$ is a homeomorphism. From the ODE above, we now also see that $X^*(x,t,\tau) = X(X^{-1}(x,t), t - \tau)$ for all $0\leq \tau \leq t$. Now from \eqref{eq:ulinftybound} we see that for $0\leq t_1, t_2 \leq T$
\begin{align}\label{eq:Xtimediff}
\abs{X(x,t_1) - X(x,t_2)} \lesssim_{\Omega,C_0 } \abs{t_1 - t_2}
\end{align}
Now let $K \subset \Omega$ be a compact set. From \eqref{eq:Xholdercontinuity} we see that for $0\leq t_1 \leq t_2 \leq T$, there exists a constant $\gamma = \gamma(K, \Omega, T, C_0 ) > 0$ so that for all $x \in K$
\begin{align}\label{eq:Xinvtimecontinuity}
\begin{aligned}
\abs{X^{-1}(x,t_1) - X^{-1}(x,t_2)} & = \abs{X^{-1}(x,t_1) - X^{-1}(X^*(x, t_2, t_2 -t_1) ,t_1)} \\
& \lesssim_\Omega \abs{x - X^*(x, t_2, t_2 -t_1)}^\gamma \\
& \lesssim_{K, \Omega, T, C_0 } \abs{t_2 - t_1}^\gamma
\end{aligned}
\end{align}
Therefore from \eqref{eq:Xholdercontinuity}, \eqref{eq:Xtimediff} and \eqref{eq:Xinvtimecontinuity} we see that $X,X^{-1}:\Omega\times[0,\infty) \to \Omega$ are continuous. 

Let us now prove that $X(\cdot,t)$ is measure preserving. Note that as the velocity is only locally log-Lipschitz, the flow map $X(\cdot,t)$ is not locally Lipschitz but is only locally H\"older continuous. Let $T>0$ and let $K \subset \Omega$ be compact. Let $U_1 \subset \Omega$ be an open set such that $\Ubar_1$ is compact and $X(X^{-1}(x,t_1), t_2) \in U_1$ for all $x \in K$ and $t_1, t_2 \in [0,T]$.  Similarly let $U_2$ be an open set with $\Ubar_1 \subset U_2 \subset \Ubar_2 \subset \Omega$ with $\Ubar_2$ compact and $X(X^{-1}(x,t_1), t_2) \in U_2$ for all $x \in \Ubar_1$ and $t_1, t_2 \in [0,T]$. Let $0<\ep_0<\half$ be small enough so that $x + 2\ep_0 y \in \Omega$ for all $x \in \Ubar_2, y \in B_1(0)$. Let $\varphi$ be a smooth bump function with $\varphi \geq 0$, $\int_{\Rsp^2} \varphi(y) \diff y = 1$ and $\supp(\varphi) \subset B_1(0)$. For $x \in \Ubar_2$ and $0<\ep \leq \ep_0$ we define $u_\ep: \Ubar_2 \times [0,T] \to \Csp$ as
\begin{align*}
u_\ep(x,t) = \int_\Omega u(x - \ep y)\varphi(y) \diff y
\end{align*} 
Observe that we have the estimates
\begin{align}\label{eq:uepestimates}
\begin{aligned}
\abs{u_\ep(x,t)} & \lesssim_{\Omega, C_0 } 1  \qq \tx{ for all } (x,t) \in \Ubar_2\times [0,T] \\
\abs{u_\ep(x_1,t) - u_\ep(x_2,t)} & \lesssim_{\Omega, \Ubar_2, \ep, C_0 } \abs{x_1 - x_2} \qq \tx{ for all } x_1, x_2 \in \Ubar_2, t \in [0,T] \\
\abs{u(x,t) - u_\ep(x,t)} & \lesssim_{\Omega, \Ubar_2, C_0 } \ep^\half  \qq \tx{ for all } (x,t) \in \Ubar_2\times [0,T]
\end{aligned}
\end{align}
Now for $x \in \Ubar_1$, consider the ODE
\begin{align}\label{eq:ODEXep}
\frac{\diff X_\ep(x,t)}{\diff t} = u_\ep(X_\ep(x,t),t), \qq X_\ep(x,0) = x
\end{align}
Clearly $X_\ep(x,t) \in U_2$, at least if $t>0$ is small enough. Now for $x \in \Ubar_1$ we have
\begin{align*}
\abs{\frac{\diff  }{\diff t} (\abs{X(x,t) - X_\ep(x,t)} + \ep^\half)} & \leq \abs{u(X(x,t),t) - u(X_\ep(x,t),t)} \\
& \quad + \abs{u(X_\ep(x,t),t) - u_\ep(X_\ep(x,t),t)} \\
& \lesssim_{\Omega, \Ubar_2,  C_0 } \phi(\abs{X(x,t) - X_\ep(x,t)}) + \ep^\half \\
& \lesssim_{\Omega, \Ubar_2, C_0 } \phi(\abs{X(x,t) - X_\ep(x,t)} + \ep^\half)
\end{align*}
Hence from \lemref{lem:phi} we see that there exists $\gamma = \gamma(\Omega, \Ubar_2, C_0, T) > 0$ so that
\begin{align}\label{eq:XcloseXep}
\abs{X(x,t) - X_\ep(x,t)} \lesssim_{\Omega} \ep^\gamma
\end{align}
In particular this implies that if $\ep_0$ is small enough, then $X_\ep(x,t) \in U_2$ for all $(x,t) \in \Ubar_1\times [0,T]$. Now we can similarly define $X^*_\ep(x,t,\tau)$ where  $X^*(x,t,\tau)$ satisfies \eqref{eq:Xstarxttau}. By the same analysis as above and using the fact that $X^*(x,t,\tau) = X(X^{-1}(x,t), t - \tau)$ for all $0\leq \tau \leq t$, we then see that for $\ep_0$ small enough, we have $X^{-1}_\ep(x,t) \in U_2$ for all $(x,t) \in \Ubar_1 \times [0,T]$.  This also implies that for $\ep_0$ small enough, $X_\ep(x,t), X^{-1}_\ep(x,t) \in U_1$ for all $(x,t) \in K\times [0,T]$. 

Now from \eqref{eq:uepestimates} and \eqref{eq:ODEXep} we see that for all $x_1, x_2 \in K$ and $t \in [0,T]$
\begin{align}\label{eq:flowLipschitzep}
\abs{X_\ep(x_1,t) - X_\ep(x_2,t)}  \lesssim_{\Omega, \Ubar, \ep, C_0, T } \abs{x_1 - x_2}
\end{align}
Moreover as $\grad \cdot u(\cdot,t) = 0$ in $\Omega$ for all $t \in [0,T]$, we see that $\grad \cdot u_\ep(\cdot,t) = 0$ in $U_2$. Therefore using \eqref{eq:flowLipschitzep} this implies that $X_\ep(\cdot,t), X^{-1}_\ep(\cdot,t): U_1 \to U_2$ are measure preserving. Now let $f:\Omega \to \Rsp$ be continuous with $\supp(f) \subset K$. Then  
\begin{align*}
\int_\Omega f(x) \diff x = \int_{K} f(x) \diff x = \int_{X_\ep^{-1}(K,t)} f(X_\ep(x,t)) \diff x = \int_{U_1} f(X_\ep(x,t)) \diff x
\end{align*}
Letting $\ep \to 0$ and using the fact that $f$ is continuous and \eqref{eq:XcloseXep}, we see that 
\begin{align*}
\int_\Omega f(x) \diff x = \int_{U_1} f(X(x,t)) \diff x = \int_\Omega f(X(x,t)) \diff x
\end{align*}
As $K$ is an arbitrary compact set, this shows that $X(\cdot,t), X^{-1}(\cdot,t):\Omega \to \Omega$ are measure preserving. 

\item As $X(x,t) \in \Omega$ for all $(x,t) \in \Omega \times [0,\infty)$, from Lemma 3.1 of \cite{HaZl21} we see that $\w(X(x,t),t) = \w_0(x)$ for a.e. $(x,t) \in \Omega \times [0,\infty)$. Hence $\w(x,t) = \w_0(X^{-1}(x,t))$ for a.e. $(x,t) \in \Omega\times[0,\infty)$. We now redefine $\w$ on set of zero measure in $\Omega \times [0,\infty)$ so that $\w(x,t) = \w_0(X^{-1}(x,t))$ holds for all $(x,t) \in \Omega\times [0,\infty)$. Now redefine $u$ on $\Omega \times [0,\infty)$ by $u(\cdot,t) = \grad^\perp \psi(\cdot,t)$ for $t \in [0,\infty)$, where $\psi(\cdot,t)$ solves \eqref{eq:psisolvingPoisson}. It is now easy to see that $(u,\w)$ satisfies the stated properties. 

\item From \eqref{eq:Xinvtimecontinuity} we see that $X^{-1}(\cdot,t_n) \to X^{-1}(\cdot,t)$ uniformly on compact sets of $\Omega$. Now the statement follows from this and by an approximation argument where we approximate $\w_0$ in $L^p(\Omega)$ by smooth functions with compact support.  

\item Using \lemref{lem:btil} part (2) and the formula \eqref{eq:btil} we see that for $z_1, z_2 \in \Dspbar$ and $t_1, t_2 \in [0,\infty)$
\begin{align*}
\abs*[\big]{\btil(z_1, t_1) - \btil(z_2, t_2)} & \leq \abs*[\big]{\btil(z_1, t_1) - \btil(z_2, t_1)} + \abs*[\big]{\btil(z_2, t_1) - \btil(z_2, t_2)} \\
& \lesssim_{C_0, \Omega} \phi(\abs{z_1 - z_2}) + \norm[L^4(\Omega)]{\w(\cdot, t_1) - \w(\cdot, t_2)}
\end{align*}
Hence $\btil$ is continuous on $\Dspbar \times [0,\infty)$. As the domain is bounded by a Jordan curve, by Carath\'{e}odory's theorem we see that $\Phi: \Omega \to \Dsp$ extends continuously to $\Omegabar$. Now as $b(x,t) = \btil(\Phi(x),t)$, we therefore now see that $b$ is continuous on $\Omegabar \times [0,\infty)$. The continuity of $u$ on $\Omega \times [0,\infty)$ follows directly from \eqref{eq:X} from which we see that $u(x,t) = b(x,t)\Phizbar(x)$.  

\end{enumerate}

\end{proof}

\section{New change of variable}\label{sec:changevar}

In this section we introduce our change of variable $\F$ on $\Dsp$ and prove several important estimates for it. We now define several functions:
\begin{enumerate}[label=\alph*)]
\item Let $g:[0,1] \to \Rsp$ be a smooth increasing function satisfying 
\begin{align*}
g(x) = 
\begin{cases}
0 \qq \tx{ for } 0 \leq x \leq \frac{1}{4} \\
1 \qq \tx{ for } \half \leq x \leq 1 
\end{cases}
\end{align*}
\item Define the function $G: \Dsp \to \Rsp$ by
\begin{align}\label{def:G}
G(z) = g(\abs{z})\abs{\Psi_z(z)}^2 + (1 - g(\abs{z})). 
\end{align}
$G$ is a smooth function on $\Dsp$ with $G(z) = 1$ for $\abs{z} \leq 1/4$ and $G(z) = \abs{\Psi_z(z)}^2$ for $1/2 \leq \abs{z} < 1$. Now from the first condition of \assref{ass:main} we know that $c_0 > 0$. If $c_0 \geq 1$ then $G(z) \geq 1$ for all $z \in \Dsp$. If $0<c_0 < 1$, then we have
\begin{align*}
G(z) \geq c_0^2\, g(\abs{z}) + (1 - g(\abs{z})) = c_0^2 + (1 - c_0^2)(1 - g(\abs{z})) \geq c_0^2
\end{align*}
Hence we see that for all $z \in \Dsp$ we have
\begin{align}\label{eq:Gbounds}
0< \min\cbrac{1, c_0^2} \leq G(z) \leq c_1^2 + 1 
\end{align}
\item Define $c:(0,1) \to \Rsp$ by 
\begin{align}\label{def:c}
c(r) = \frac{1}{2\pi} \int_0^{2\pi} G(re^{is}) \diff s
\end{align}
Clearly $c$ is smooth and from \eqref{eq:Gbounds} we see that for all $r \in (0,1)$ we have
\begin{align}\label{eq:cbounds}
0< \min\cbrac{1, c_0^2} \leq c(r) \leq c_1^2 + 1 
\end{align}
We also see that $c(r) = 1$ for $0< r \leq 1/4$.
\item Define $L: (0,1) \times \Rsp \to \Rsp$ by 
\begin{align}\label{def:L}
L(r,\theta) = \frac{1}{c(r)} \int_0^\theta G(re^{is}) \diff s
\end{align}
From  $\eqref{eq:cbounds}$ we see that $L$ is well defined. Observe that $L$ is smooth and for any fixed $r$, $L(r, \cdot)$ is a strictly increasing function. It is also easy to see that 
\begin{align}\label{eq:Lprop}
\begin{aligned}
L(r,0) & = 0 \qq \tx{ for all } r \in (0,1) \\
L(r, \theta + 2\pi) & = L(r,\theta) + 2\pi \qq \tx{ for all } (r,\theta) \in (0,1)\times \Rsp \\
L(r, \theta) & = \theta \qq \tx{ for all } (r,\theta) \in (0,1/4] \times \Rsp
\end{aligned}
\end{align}
\item We now define our change of variable $F:\Dsp \to \Dsp$ by 
\begin{align}\label{def:F}
F(re^{i\theta}) = 
\begin{cases}
re^{iL(r,\theta)} & \tx{ for } 0<r<1  \\
0 & \tx{ for } r = 0
\end{cases}
\end{align}
From \eqref{eq:Lprop} we see that $e^{iL(r,\theta)}$ is $2\pi$ periodic in the $\theta$ variable and hence $F$ is well defined. We also note that $F(z) = z$ for all $\abs{z} \leq 1/4$ and for $z \in [0,1)$. As $L $ is smooth, we see that $F$ is also a smooth function. 
\item We also need some additional functions. Define $a,b : \Dsp \to \Rsp$ by
\begin{align}\label{def:ab}
a(z) = \Real(\Psi_z(z)) \qq b(z) = \Imag(\Psi_z(z))
\end{align}
Now observe that as $\Psi_z : \Dsp \to \Csp$ is a bounded holomorphic function, it has a non-tangential limit at the boundary almost everywhere (see Theorem 1.3 of \cite{Du70}). Let $a_1, b_1: \Sone \to \Rsp$ be the boundary values of $a$ and $b$ respectively. Therefore $a_1, b_1 \in \Linfty(\Sone)$ with
\begin{align}\label{eq:aonebound}
\norm[L^\infty(\Sone)]{a_1}, \norm[L^\infty(\Sone)]{b_1}, \norm[L^\infty(\Omega)]{a}, \norm[L^\infty(\Omega)]{b} \leq c_1
\end{align}
and we have 
\begin{align}\label{eq:afromaone}
a(r,\cdot) = P_r\conv a_1 \qq b(r,\cdot) = P_r\conv b_1 \qq \tx{ for all } 0\leq r< 1
\end{align}
As $\Psi_z$ is holomorphic, we see that 
\begin{align*}
\Hil(a_1) = b_1 - \Avg(b_1)
\end{align*}
As explained in \secref{sec:notation} we can view $a, b: [0,1] \times \Rsp \to \Rsp$ as $2\pi$ periodic functions in the second variable. From this we see that for all $r \in [0,1]$ we have
\begin{align}\label{eq:Hila}
\Hil(a(r, \cdot)) = b(r, \cdot) - \Avg(b(r,\cdot))
\end{align}
By abuse of notation we will write the above equality as $\Hil a = b - \Avg(b)$. Taking its $\partial_s$ derivative we get for all $r \in [0,1)$
\begin{align}\label{eq:bs}
\psabs(a(r, \cdot)) = b_s(r, \cdot)
\end{align}
Again abusing notation, we will write the above equality as $\psabs a = b_s$. From condition (2) of \assref{ass:main} we see that 
\begin{align}\label{eq:abHhalf}
\sup_{r \in [0,1)} \norm[\Hhalf(\Sone)]{a(r, \cdot)}, \sup_{r \in [0,1)} \norm[\Hhalf(\Sone)]{b(r, \cdot)} \leq c_2
\end{align}
Now we know that $\Psi_z = a + ib$ is a bounded holomorphic function. As $\Psi_z^2 = a^2 - b^2 + 2iab$, we see that $\Hil(a^2 - b^2)$ is a bounded function (here we have again abused notation similarly as above). From condition (3) in  \assref{ass:main} we see that $\Hil(a^2 + b^2)$ is a bounded function. Hence we see that 
\begin{align}\label{eq:aHilLinfty}
\sup_{r \in [0,1)} \norm[\Linfty(\Sone)]{\Hil(a^2(r, \cdot))} \lesssim_\Omega 1
\end{align}
\end{enumerate}

Now we introduce polar coordinates
\begin{align*}
x & = r\cos(\theta)   & \hat{r} & = \frac{z}{\abs{z}} = e^{i\theta} \\
y & = r\sin(\theta)   & \hat{\theta} & = i\frac{z}{\abs{z}} = ie^{i\theta}
\end{align*}
We also have
\begin{align*}
\grad & = \hat{r} \frac{\partial }{\partial r} + \hat{\theta} \frac{1}{r} \frac{\partial }{\partial \theta} = \frac{\partial }{\partial x} + i\frac{\partial}{\partial y} \\
\frac{\partial }{\partial r} & = \cos(\theta)\frac{\partial }{\partial x} + \sin(\theta) \frac{\partial }{\partial y} \\
\frac{1}{r}\frac{\partial }{\partial \theta} & = -\sin(\theta)\frac{\partial }{\partial x} + \cos(\theta) \frac{\partial }{\partial y}
\end{align*}
Now as $a,b$ are the real and imaginary parts of a holomorphic function, we get the Cauchy Riemann equations
\begin{align*}
a_x = b_y \quad \tx{ and } \quad  a_y = - b_x
\end{align*} 
Writing these equations in polar coordinates we get for $0< r <1$
\begin{align}\label{eq:CauchyRiempolar}
a_r = \frac{b_\theta}{r}  \quad \tx{ and } \quad b_r = -\frac{a_\theta}{r} 
\end{align}

We now prove some estimates for the functions introduced above. The important estimates are \lemref{lem:Fbilipschitz} and \lemref{lem:FrFthetaestimates}, which are crucially used in the energy estimate in the next section. The other lemmas proved here are used to prove these two lemmas. 
\begin{lemma}\label{lem:Grestimate}
For $(r,\theta) \in (0,1) \times [0,2 \pi]$ we have
\begin{align*}
\abs{\int_0^\theta G_r(r,s) \diff s} \lesssim_\Omega 1
\end{align*}
\end{lemma}
\begin{proof}
From the definition of $G$ in \eqref{def:G}  and $a,b$ in \eqref{def:ab}, we see that for $(r,s) \in (0,1)\times \Rsp$ we have
\begin{align*}
G(r,s) = g(r)(a^2(r,s) + b^2(r,s)) + (1 - g(r))
\end{align*}
From \eqref{eq:CauchyRiempolar} we have
\begin{align}\label{eq:Grequations}
\begin{aligned}
G_r & = g'(r)(a^2 + b^2 - 1) + 2g(r)(a a_r + b b_r)  \\
& = g'(r)(a^2 + b^2 - 1) + 2\frac{g(r)}{r}(a b_s - b a_s) 
\end{aligned}
\end{align}
Now from \eqref{eq:aonebound} we see that
\begin{align*}
\abs{ \int_0^\theta g'(r)(a^2(r,s) + b^2(r,s) - 1)\diff s} \lesssim_\Omega 1 \qq \tx{ for all } (r,\theta) \in (0,1) \times [0,2 \pi]
\end{align*}
Using \eqref{eq:bs} we obtain
\begin{align*}
& \int_0^\theta (a b_s - b a_s)(r,s) \diff s \\
& = 2\int_0^\theta (a b_s)(r,s) \diff s - \cbrac{a(r,\theta)b(r, \theta) - a(r, 0)b(r, 0)} \\
& = 2\int_0^\theta (a \psabs a)(r,s) \diff s - \cbrac{a(r,\theta)b(r, \theta) - a(r, 0)b(r, 0)}
\end{align*}
We can now utilize \eqref{eq:fabsfform} in \lemref{lem:Hhalfformula} which gives a useful representation of $a \psabs a$. We get
\begin{align*}
& \int_0^\theta (a b_s - b a_s)(r,s) \diff s \\
& = \frac{1}{4\pi}\int_0^\theta \int_0^{2\pi} \frac{(a(r,s) - a(r,\bp))^2}{\sin^2\brac{\frac{s - \bp}{2}}} \diff \bp \diff s + \cbrac{(\Hil(a^2))(r, \theta) - (\Hil(a^2))(r, 0)} \\
& \quad  - \cbrac{a(r,\theta)b(r, \theta) - a(r, 0)b(r, 0)}
\end{align*}
Hence from \eqref{eq:fhalf2normform} of \lemref{lem:Hhalfformula}, along with  \eqref{eq:abHhalf}, \eqref{eq:aHilLinfty} and \eqref{eq:aonebound} we see that 
\begin{align}\label{eq:absminusbasint}
\abs{\int_0^\theta (a b_s - b a_s)(r,s) \diff s} \lesssim_\Omega 1  \qq \tx{ for all } (r,\theta) \in (0,1) \times [0,2 \pi]
\end{align}
Therefore proved. 
\end{proof}

\begin{lemma}
We have the following estimates:
\begin{enumerate}
\item For $r \in (0,1)$ we have
\begin{align}\label{eq:cderivest}
\abs{c'(r)} \lesssim_\Omega 1
\end{align}
\item For $(r, \theta) \in (0,1) \times [0,2\pi]$ we have
\begin{align}\label{eq:LrLthetaest}
\abs{L_\theta(r, \theta)} + \abs{L_r(r, \theta)} + \abs{F_\theta(r, \theta)} + \abs{F_r(r, \theta)} \lesssim_\Omega 1
\end{align}
\end{enumerate}
\end{lemma}
\begin{proof}
From the definition of $c$ in \eqref{def:c} and \lemref{lem:Grestimate} we easily get that $\abs{c'(r)} \lesssim_\Omega 1$ for all $r \in (0,1)$. Now from the definition of $L$ in \eqref{def:L} and the estimates \eqref{eq:Gbounds}, \eqref{eq:cbounds} we see that
\begin{align*}
\abs{L_\theta(r,\theta)} = \abs{\frac{G(r,\theta)}{c(r)}} \lesssim_\Omega 1 \qq \tx{ for all } (r, \theta) \in (0,1) \times [0,2\pi]
\end{align*}
We also observe that for $0< r < 1$
\begin{align*}
L_r(r, \theta) = - \frac{c'(r)}{c(r)^2} \int_0^\theta G(r,s) \diff s + \frac{1}{c(r)} \int_0^\theta G_r(r,s) \diff s 
\end{align*}
Hence from \eqref{eq:cderivest}, \eqref{eq:Gbounds}, \eqref{eq:cbounds} and \lemref{lem:Grestimate} we get the required estimate for $\abs{L_r(r, \theta)}$. To get the estimates for the derivatives of $F$, observe that for $0<r<1$ we have
\begin{align}\label{eq:FthetaFr}
\begin{aligned}
F_\theta(r,\theta) & = ire^{iL(r,\theta)}L_\theta(r,\theta) \\
F_r(r, \theta) & = e^{iL(r, \theta)} + ire^{iL(r,\theta)}L_r(r,\theta)
\end{aligned}
\end{align}
The estimates for $F_\theta(r,\theta)$ and $F_r(r,\theta)$ now follow easily. 
\end{proof}

The following lemma shows that $F:\Dsp \to \Dsp$ is bi-Lipschitz and this fact is crucially used in the next section to prove the energy estimate. 

\begin{lemma}\label{lem:Fbilipschitz}
For all $x, y \in \Dsp$ we have
\begin{align*}
\abs{x- y} \lesssim_\Omega \abs{F(x) - F(y)} \lesssim_\Omega \abs{x - y}
\end{align*}
\end{lemma}
\begin{proof}
As $F(z) = z$ for $\abs{z} \leq 1/4$ and from \eqref{eq:LrLthetaest} we see that $\abs{\grad F(z)} \lesssim_\Omega 1$ for all $z \in \Dsp$. Hence $\abs{F(x) - F(y)} \lesssim_\Omega \abs{x - y}$ for all $x,y \in \Dsp$. From the definition of $F$ it is also clear that $F$ is a smooth homeomorphism on $\Dsp$. Let $J_F$ be the Jacobian of the map $F$ and so we have for $0<\abs{z} < 1$
\begin{align*}
\det(J_F)(z) = 
\begin{vmatrix}
\Real(F_x) & \Real(F_y) \\
\Imag(F_x) & \Imag(F_y)
\end{vmatrix}
= \Imag(\overline{F_x} F_y) = \frac{1}{r} \Imag(\overline{F_r} F_\theta)
\end{align*}
Now using \eqref{eq:FthetaFr} and the estimates \eqref{eq:Gbounds} and \eqref{eq:cbounds} we see that 
\begin{align*}
\abs{\det(J_F)}(r, \theta) = L_\theta(r, \theta) = \frac{G(r, \theta)}{c(r)} \approx_\Omega 1 \qq \tx{ for all } (r, \theta) \in (0,1) \times [0,2\pi]
\end{align*}
As $F(z) = z$ for $\abs{z} \leq 1/4$, this implies that $\abs{\grad F^{-1}(z)} \lesssim_\Omega 1$ for all $z \in \Dsp$. Hence  $\abs{x- y} \lesssim_\Omega \abs{F(x) - F(y)}$ for all $x, y \in \Dsp$. 

\end{proof}

\begin{lemma}\label{lem:athetaLinftyLtwo}
For $r \in (\frac{1}{100}, 1)$ we have
\begin{align*}
\norm[\Linfty(\Sone)]{a_\theta(r, \cdot)}, \norm[\Linfty(\Sone)]{b_\theta(r, \cdot)}, \norm[\Ltwo(\Sone)]{a_\theta(r, \cdot)}^2, \norm[\Ltwo(\Sone)]{b_\theta(r, \cdot)}^2 \lesssim_\Omega \frac{1}{1-r}
\end{align*}
\end{lemma}
\begin{proof}
Using the fact that $a$ and $b$ are harmonic functions on $\Dsp$ and the estimates \eqref{eq:aonebound} and \eqref{eq:abHhalf}, the estimates follow from basic properties of convolution by the Poisson kernel \eqref{eq:afromaone}. 
\end{proof}

\begin{lemma}\label{lem:Grrintest}
For $(r, \theta) \in (\frac{1}{100},1) \times [0,2\pi]$ we have
\begin{align*}
\abs{\int_0^\theta G_{rr}(r,s) \diff s} \lesssim_\Omega \frac{1}{1-r}
\end{align*}
\end{lemma}
\begin{proof}
From \eqref{eq:Grequations} we see that 
\begin{align*}
G_r = g'(r)(a^2 + b^2 - 1) + 2\frac{g(r)}{r}(a b_s - b a_s) 
\end{align*}
Now taking another $\partial_r$ derivative and using $\eqref{eq:CauchyRiempolar}$ we get
\begin{align*}
G_{rr} & = g''(r)(a^2 + b^2 - 1) + \brac{4\frac{g'(r)}{r}  -2\frac{g(r)}{r^2}} (ab_s - ba_s) \\
& \quad  + 2 \frac{g(r)}{r^2}(b_s^2 - aa_{ss} + a_s^2 - bb_{ss})
\end{align*}
From \eqref{eq:aonebound} and \eqref{eq:absminusbasint} we have for $(r, \theta) \in (\frac{1}{100},1) \times [0,2\pi]$
\begin{align*}
\abs{g''(r)\int_0^\theta (a^2 + b^2 - 1) \diff s} & \lesssim_\Omega 1 \\
\abs{\brac{4\frac{g'(r)}{r}  -2\frac{g(r)}{r^2}}\int_0^\theta (ab_s - ba_s) \diff s} & \lesssim_\Omega 1
\end{align*}
To control the last term we observe that
\begin{align*}
& \int_0^\theta (b_s^2 - aa_{ss} + a_s^2 - bb_{ss}) \diff s \\
& = 2\int_0^\theta (a_s^2 + b_s^2) \diff s  - \cbrac{a(r,\theta)a_s(r,\theta) - a(r,0)a_s(r,0) + b(r,\theta)b_s(r,\theta) - b(r,0)b_s(r,0)}
\end{align*}
Now using \lemref{lem:athetaLinftyLtwo} we see that for all  $(r, \theta) \in (\frac{1}{100},1) \times [0,2\pi]$
\begin{align*}
\abs{2 \frac{g(r)}{r^2} \int_0^\theta (b_s^2 - aa_{ss} + a_s^2 - bb_{ss}) \diff s} \lesssim_\Omega \frac{1}{1-r}
\end{align*}
Hence proved. 
\end{proof}

\begin{lemma}
We have the following estimates:
\begin{enumerate}
\item For $r \in (\frac{1}{100},1)$ we have
\begin{align}\label{eq:csecondderivest}
\abs{c''(r)} \lesssim_\Omega \frac{1}{1-r}
\end{align}
\item For $(r, \theta) \in (\frac{1}{100},1) \times [0,2\pi]$ we have
\begin{align}\label{eq:Lrrest}
\abs{G_\theta(r,\theta)}, \abs{G_r(r,\theta)}, \abs{L_{\theta\theta}(r, \theta)} + \abs{L_{\theta r}(r, \theta)} + \abs{L_{rr}(r, \theta)}  \lesssim_\Omega \frac{1}{1-r}
\end{align}
\item For $(r, \theta) \in (\frac{1}{100},1) \times [0,2\pi]$ we have
\begin{align}\label{eq:Frrest}
\abs{F_{\theta\theta}(r, \theta)} + \abs{F_{\theta r}(r, \theta)} + \abs{F_{rr}(r, \theta)}  \lesssim_\Omega \frac{1}{1-r}
\end{align}
\end{enumerate}
\end{lemma}
\begin{proof}
The first estimate \eqref{eq:csecondderivest} follows directly from the definition of $c$ in \eqref{def:c} and \lemref{lem:Grrintest}. Now from the definition of $G$ in \eqref{def:G} and \eqref{eq:CauchyRiempolar} we see that
\begin{align*}
G_\theta(r,\theta) & = 2g(r)(aa_\theta + bb_\theta) \\
 G_r(r,\theta) & = g'(r)(a^2 + b^2 - 1) + 2\frac{g(r)}{r}(a b_s - b a_s)
\end{align*}
Hence we get the required bounds on $\abs{G_r}$ and $\abs{G_\theta}$ from \lemref{lem:athetaLinftyLtwo}. Now from the definition of $L$ in \eqref{def:L} we see that 
\begin{align*}
L_{\theta\theta}(r,\theta) &  = \frac{G_\theta(r,\theta)}{c(r)} \\
L_{\theta r}(r, \theta) & = - \frac{c'(r)}{c^2(r)}G(r, \theta) + \frac{G_r(r, \theta)}{c(r)} 
\end{align*}
The estimates for $\abs{L_{\theta\theta}}$ and $\abs{L_{\theta r}}$ now follow from the estimates for $\abs{G_r}$, $\abs{G_\theta}$ and the estimates \eqref{eq:cderivest} and \eqref{eq:cbounds}. We also observe that
\begin{align*}
L_{rr}(r, \theta)  = \brac{\frac{1}{c(r)}}''\int_0^\theta G(r, s) \diff s +  2\brac{\frac{1}{c(r)}}'\int_0^\theta G_r(r, s) \diff s +  \nobrac{\frac{1}{c(r)}}\int_0^\theta G_{rr}(r, s) \diff s
\end{align*}
The estimate for $\abs{L_{rr}}$ now follows from \lemref{lem:Grestimate}, \eqref{eq:cderivest}, \lemref{lem:Grrintest} and \eqref{eq:csecondderivest}. Similarly from the formulae of $F_\theta$ and $F_r$ in \eqref{eq:FthetaFr}, we easily get the required estimates for $\abs{F_{\theta\theta}}, \abs{F_{\theta r}}$ and $\abs{F_{rr}} $ from the estimates \eqref{eq:LrLthetaest} and \eqref{eq:Lrrest}. 
\end{proof}

\begin{lemma}\label{lem:FrFthetaestimates}
We have the following estimates:
\begin{enumerate}
\item For $z_1, z_2 \in \Dsp$ with $\frac{1}{100} < \abs{z_1}, \abs{z_2} < 1$ we have
\begin{align}\label{eq:Fthetamainuseful}
\abs{\brac{\frac{\abs{\Psi_z(z_1)}^{-2}}{\abs{z_1}} F_\theta(z_1) } - \brac{\frac{\abs{\Psi_z(z_2)}^{-2}}{\abs{z_2}} F_\theta(z_2) }} \lesssim_\Omega \abs{z_1 - z_2}
\end{align}
\item For $z_1, z_2 \in \Dsp$ with $\frac{1}{100} < \abs{z_1}, \abs{z_2} < 1$ we have
\begin{align}\label{eq:Frmainuseful}
\abs{\brac{\abs{\Psi_z(z_1)}^{-2} F_r(z_1) } - \brac{\abs{\Psi_z(z_2)}^{-2} F_r(z_2) }} \lesssim_\Omega \frac{\abs{z_1 - z_2}}{\min(1- \abs{z_1}, 1 - \abs{z_2})}
\end{align}
\end{enumerate}
\end{lemma}
\begin{proof}
From \eqref{eq:FthetaFr} we see that for $z \in \Dsp$, $z \neq 0$
\begin{align*}
\frac{\abs{\Psi_z(z)}^{-2}}{\abs{z}} F_\theta(z) & = i\abs{\Psi_z(z)}^{-2}e^{iL(r,\theta)}L_\theta(r,\theta)\\
 & = i\abs{\Psi_z(z)}^{-2}e^{iL(r,\theta)}\frac{G(r,\theta)}{c(r)} \\
 & = i \frac{e^{iL(r,\theta)}}{c(r)}\brac{g(r) + \frac{(1 - g(r))}{\abs{\Psi_z(z)}^2}}
\end{align*}
Therefore using the fact that $c_0 \leq \abs{\Psi_z(z)} $ and the bound \eqref{eq:cbounds} we get
\begin{align*}
& \abs{\brac{\frac{\abs{\Psi_z(z_1)}^{-2}}{\abs{z_1}} F_\theta(z_1) } - \brac{\frac{\abs{\Psi_z(z_2)}^{-2}}{\abs{z_2}} F_\theta(z_2) }} \\
& \lesssim_\Omega \abs{e^{iL(r_1, \theta_1)} - e^{iL(r_2, \theta_2)}} + \abs{c(r_1) - c(r_2)} + \abs{g(r_1) - g(r_2)} + \abs{\frac{1 - g(r_1)}{\abs{\Psi_z(z_1)}^2} - \frac{1 - g(r_2)}{\abs{\Psi_z(z_2)}^2} }
\end{align*}
Note that although the function $\frac{1}{\abs{\Psi_z(z)}^2}$ is not  Lipschitz on $\Dsp$, the function $\frac{1 - g(r)}{\abs{\Psi_z(z)}^2}$ is indeed Lipschitz on $\Dsp$ due to the fact that $(1 - g(r))$ is zero for $\frac{1}{2} \leq r < 1$. Therefore by additionally using the estimates \eqref{eq:cderivest}, \eqref{eq:LrLthetaest}, the first estimate  \eqref{eq:Fthetamainuseful} now follows. For the second estimate we use the bound $c_0 \leq \abs{\Psi_z(z)} $ and \eqref{eq:LrLthetaest} to obtain
\begin{align*}
& \abs{\brac{\abs{\Psi_z(z_1)}^{-2} F_r(z_1) } - \brac{\abs{\Psi_z(z_2)}^{-2} F_r(z_2) }} \\
& \lesssim_\Omega \abs{\abs{\Psi_z(z_1)}^2 - \abs{\Psi_z(z_2)}^2} + \abs{F_r(z_1) - F_r(z_2)}
\end{align*}
The estimate \eqref{eq:Frmainuseful} now follows easily from the fact that $\abs{\Psi_z(z)}^2 = a^2(z) + b^2(z)$ in conjunction with  \eqref{eq:CauchyRiempolar},  \lemref{lem:athetaLinftyLtwo} and \eqref{eq:Frrest}. 
\end{proof}

\section{Proof of main results}\label{sec:energy}

In this section we prove \thmref{thm:main} and \propref{prop:main}. Let us briefly explain how the change of variable $F$ helps in proving uniqueness. As explained in the introduction, the differential equation of $Y(y,t)$ is 
\begin{align*}
\frac{\diff Y(y,t)}{\diff t} = \btil(Y(y,t),t)\abs{\Psi_z(Y(y,t))}^{-2}
\end{align*}
and as  $\abs{\Psi_z}^{-2}$ is not log-Lipschitz near the boundary if the domain is less regular than $C^{1,1}$, one does not have a good bound on $\pt \abs{Y_1(y,t) - Y_2(y,t)}$. Hence the energy $E_2(t)$ defined in \eqref{eq:EnergystandardY} is not good enough to prove uniqueness. However the differential equation for $F(Y(y,t))$ is 
\begin{align*}
 \frac{\diff F(Y(y,t))}{\diff t} & = \brac{\btil(Y(y,t),t)\abs{\Psi_z(Y(y,t))}^{-2} } \cdot \grad F(Y(y,t)) \\
& = \Real\cbrac{\btil(Y(y,t),t)\frac{\Ybar(y,t)}{\abs{Y(y,t)}} } \abs{\Psi_z(Y(y,t))}^{-2} \partial_r F(Y(y,t)) \\
& \quad  + \Imag\cbrac{\btil(Y(y,t),t)\frac{\Ybar(y,t)}{\abs{Y(y,t)}} }\frac{\abs{\Psi_z(Y(y,t))}^{-2}}{\abs{Y(y,t)}}\partial_\theta F(Y(y,t)) 
\end{align*}
Now as we explain in more detail below, using the properties proved of $F$ and $\btil$ one can show that the terms on the right hand side of the above equation are log-Lipschitz! This removes the fundamental obstruction to proving uniqueness. Let us now us now prove \thmref{thm:main} in full detail. 

\begin{proof}[Proof of \thmref{thm:main}]
Let $(u_1, \w_1)$ and $(u_2, \w_2)$ be two Yudovich weak solutions in $\Omega$ in the time interval $[0,\infty)$ with the same initial vorticity $\w_0 \in \Linfty(\Omega)$. Let $Y_1, Y_2: \Dsp \times [0,\infty) \to \Dsp $ be the corresponding flows in $\Dsp$ of the solutions given by \eqref{eq:Y}. Consider the energy
\begin{align}\label{def:Energy}
E(t) = \int_\Dsp \abs{F(Y_1(y,t)) - F(Y_2(y,t)) } \abs{\Psi_z(y)}^2 \diff y
\end{align} 
where $F:\Dsp \to \Dsp$ is the change of variable defined in \eqref{def:F}. Our aim is to prove an estimate of the form $E'(t) \lesssim_{\norm[\infty]{\w_0}, \Omega} \phi(E(t)) $ on a suitable time interval. Once we have such an estimate, uniqueness will follow from Gr\"onwall's inequality.

From \eqref{eq:Y} and \lemref{lem:btil} part (1) we see that $\abs{\pt Y_i(y,t)} \lesssim_{\Omega, \norm[\infty]{\w_0}} 1 $ for all $(y,t) \in \Dsp \times [0,\infty)$ and $i = 1,2$. Using this and \lemref{lem:Fbilipschitz} we see that there exists $T = T(\norm[\infty]{\w_0}, \Omega) > 0$ small enough so that 
\begin{align}
\abs{Y_i(y,t) - y} & \leq \frac{1}{16} \qq \tx{ for all } (y, t) \in \Dsp\times [0,T], i = 1,2 \label{eq:Yiclosetoy} \\ 
\abs{F(Y_1(y,t)) - F(Y_2(y,t))} & \leq \frac{1}{10} \qq \tx{ for all } (y, t) \in \Dsp\times [0,T] \label{eq:FYoneminusFYtwoboundtime}
\end{align} 
In particular this means from \eqref{eq:Yiclosetoy} that for all $t \in [0,T]$ and $i = 1,2$
\begin{align}\label{eq:boundYi}
\begin{aligned}
\abs{Y_i(y,t)} & \leq \frac{1}{4} \qq \tx{ for } y \in \Dsp, \abs{y} \leq \frac{1}{8} \\
\abs{Y_i(y,t)} & \geq \frac{1}{16}  \qq \tx{ for } y \in \Dsp, \abs{y} \geq \frac{1}{8}
\end{aligned}
\end{align}
From now on we will suppress the variables $(y,t)$ in $Y_i(y,t)$ and simply write $Y_i$ for easier readability. Similarly we will suppress the time variable $t$ in $\btil(y,t)$.  Now using \eqref{eq:Y} gives us
\begin{align*}
E'(t) & \leq \int_\Dsp \abs{\pt F(Y_1) - \pt F(Y_2) } \abs{\Psi_z(y)}^2 \diff y \\
& = \int_\Dsp \abs{ \brac{\frac{\diff Y_1}{\diff t}}\cdot \grad F(Y_1) - \brac{\frac{\diff Y_2}{\diff t}}\cdot \grad F(Y_2) } \abs{\Psi_z(y)}^2 \diff y \\
& = \int_\Dsp \abs{ \brac{\btil_1(Y_1)\abs{\Psi_z(Y_1)}^{-2} }\cdot \grad F(Y_1) - \brac{\btil_2(Y_2)\abs{\Psi_z(Y_2)}^{-2} } \cdot \grad F(Y_2) } \abs{\Psi_z(y)}^2 \diff y
\end{align*}
Using the fact that $F(z) = z$ for all $\abs{z} \leq 1/4$ and by using \eqref{eq:boundYi} we see that for all $t \in [0,T]$ we have
\begin{align*}
& E'(t) \\
& \leq \int_{\abs{y} \leq \frac{1}{8}} \abs{ \brac{\btil_1(Y_1)\abs{\Psi_z(Y_1)}^{-2} } - \brac{\btil_1(Y_2)\abs{\Psi_z(Y_2)}^{-2} }} \abs{\Psi_z(y)}^2 \diff y \\
& \quad + \int_{\frac{1}{8} < \abs{y} < 1} \abs{ \brac{\btil_1(Y_1)\abs{\Psi_z(Y_1)}^{-2} }\cdot \grad F(Y_1) - \brac{\btil_1(Y_2)\abs{\Psi_z(Y_2)}^{-2} } \cdot \grad F(Y_2) } \abs{\Psi_z(y)}^2 \diff y \\
& \quad + \int_\Dsp \abs{ \brac{(\btil_1 - \btil_2)(Y_2)\abs{\Psi_z(Y_2)}^{-2} }\cdot \grad F(Y_2) } \abs{\Psi_z(y)}^2 \diff y \\
& = \rom{1} + \rom{2} + \rom{3}
\end{align*}
Let us now control each of the terms. Among these terms, the first term will be easily controlled as it is localized to the interior. The second term is nontrivial and we will use the properties of $F$ and $\btil$ proved earlier to control this. The third term will be easily controlled by standard arguments as $F$ is Lipschitz. 

\smallskip
\noindent$\mathbf{Controlling \enspace \rom{1}:}$ For the first term we first observe from \eqref{eq:boundYi} that $\abs{Y_i(y,t)} \leq \frac{1}{4}$ for $\abs{y} \leq \frac{1}{8}$. Now by using the fact that $\abs{\Psi_z(z)}^{-2}$ is Lipschitz on $\abs{z} \leq \frac{1}{4}$, and the fact that $\btil_1$ is log-Lipschitz from \lemref{lem:btil} part (2), we easily see that 
\begin{align}\label{eq:Iestimated}
\rom{1} \lesssim_{\norm[\infty]{\w_0}, \Omega} \int_{\abs{y} \leq \frac{1}{8}} \phi\brac{\abs{Y_1 - Y_2}}  \abs{\Psi_z(y)}^2 \diff y
\end{align}

\noindent$\mathbf{Controlling \enspace \rom{2}:}$ For the second term, we first observe that for $\frac{1}{8} \leq \abs{y} < 1$ we have from \eqref{eq:boundYi} that $\abs{Y_i(y,t)} \geq \frac{1}{16}$. Hence
\begin{align*}
& \brac{\btil_1(Y_i)\abs{\Psi_z(Y_i)}^{-2} } \cdot \grad F(Y_i) \\
& = \Real\cbrac{\btil_1(Y_i)\frac{\Ybar_i}{\abs{Y_i}}  }\abs{\Psi_z(Y_i)}^{-2}\partial_r F(Y_i) + \Imag\cbrac{\btil_1(Y_i)\frac{\Ybar_i}{\abs{Y_i}} }\frac{\abs{\Psi_z(Y_i)}^{-2}}{\abs{Y_i}}\partial_\theta F(Y_i) 
\end{align*}
Now using \lemref{lem:btil} part (1) and the $F_\theta$ bound from \eqref{eq:LrLthetaest} we see that 
\begin{align*}
& \abs{\Imag\cbrac{\btil_1(Y_1)\frac{\Ybar_1}{\abs{Y_1}} }\frac{\abs{\Psi_z(Y_1)}^{-2}}{\abs{Y_1}}\partial_\theta F(Y_1) - \Imag\cbrac{\btil_1(Y_2)\frac{\Ybar_2}{\abs{Y_2}} } \frac{\abs{\Psi_z(Y_2)}^{-2}}{\abs{Y_2}}\partial_\theta F(Y_2) } \\
& \lesssim_{\norm[\infty]{\w_0}, \Omega} \abs{\btil_1(Y_1) - \btil_1(Y_2)} + \abs{Y_1 - Y_2} + \abs{\frac{\abs{\Psi_z(Y_1)}^{-2}}{\abs{Y_1}}\partial_\theta F(Y_1) - \frac{\abs{\Psi_z(Y_2)}^{-2}}{\abs{Y_2}}\partial_\theta F(Y_2)} \\
&  \lesssim_{\norm[\infty]{\w_0}, \Omega} \phi(\abs{Y_1 - Y_2})
\end{align*}
where in the last step we used \lemref{lem:btil} part (2) and \eqref{eq:Fthetamainuseful}. Therefore  we see that for $\frac{1}{8} \leq \abs{y} < 1$
\begin{align*}
& \abs{  \brac{\btil_1(Y_1)\abs{\Psi_z(Y_1)}^{-2} } \cdot \grad F(Y_1) -  \brac{\btil_1(Y_2)\abs{\Psi_z(Y_2)}^{-2} } \cdot \grad F(Y_2)} \\
& \lesssim_{\norm[\infty]{\w_0}, \Omega} \abs{ \Real\cbrac{\btil_1(Y_1)\frac{\Ybar_1}{\abs{Y_1}} }\abs{\Psi_z(Y_1)}^{-2}\partial_r F(Y_1) -  \Real\cbrac{\btil_1(Y_2)\frac{\Ybar_2}{\abs{Y_2}} }\abs{\Psi_z(Y_2)}^{-2}\partial_r F(Y_2)} \\
& \quad + \phi\brac{\abs{Y_1 - Y_2}}
\end{align*}
Now for fixed $(y,t)$ with $\frac{1}{8} \leq \abs{y} < 1$ we assume without loss of generality that $1 - \abs{Y_1(y,t)} \leq 1 - \abs{Y_2(y,t)}$. Hence using \lemref{lem:btil} part (2) and the bound on $F_r$ from \eqref{eq:LrLthetaest} we see that
\begin{align*}
& \abs{  \brac{\btil_1(Y_1)\abs{\Psi_z(Y_1)}^{-2} } \cdot \grad F(Y_1) -  \brac{\btil_1(Y_2)\abs{\Psi_z(Y_2)}^{-2} } \cdot \grad F(Y_2)} \\
& \lesssim_{\norm[\infty]{\w_0}, \Omega} \abs{ \Real\cbrac{\btil_1(Y_1)\frac{\Ybar_1}{\abs{Y_1}} } \brac{\abs{\Psi_z(Y_1)}^{-2}\partial_r F(Y_1) - \abs{\Psi_z(Y_2)}^{-2}\partial_r F(Y_2)  }} \\
& \quad + \abs{\Real\cbrac{\btil_1(Y_1)\frac{\Ybar_1}{\abs{Y_1}} } - \Real\cbrac{\btil_1(Y_2)\frac{\Ybar_2}{\abs{Y_2}} }} \, \abs{\Psi_z(Y_2)}^{-2}\abs{\partial_r F(Y_2)} +  \phi\brac{\abs{Y_1 - Y_2}} \\
&  \lesssim_{\norm[\infty]{\w_0}, \Omega} \abs{ \Real\cbrac{\btil_1(Y_1)\frac{\Ybar_1}{\abs{Y_1}} } \brac{\abs{\Psi_z(Y_1)}^{-2}\partial_r F(Y_1) - \abs{\Psi_z(Y_2)}^{-2}\partial_r F(Y_2)  }}  + \phi\brac{\abs{Y_1 - Y_2}}
\end{align*}
To estimate the first term, there are two cases:

\smallskip
\textbf{Case 1}: $\abs{Y_1(y,t) - Y_2(y,t)} < 1 - \abs{Y_1(y,t)}$. 

As $\frac{1}{8} \leq \abs{y} < 1$, this implies that we have $\abs{\ln(\abs{Y_1(y,t) - Y_2(y,t)})} \geq \abs{\ln(1 - \abs{Y_1(y,t)})}$. Therefore using \eqref{eq:Frmainuseful} and \lemref{lem:btil} part (4) we obtain
\begin{align*}
& \abs{ \Real\cbrac{\btil_1(Y_1)\frac{\Ybar_1}{\abs{Y_1}} } \brac{\abs{\Psi_z(Y_1)}^{-2}\partial_r F(Y_1) - \abs{\Psi_z(Y_2)}^{-2}\partial_r F(Y_2)  }} \\
&  \lesssim_{\norm[\infty]{\w_0}, \Omega} \phi(1 - \abs{Y_1}) \frac{\abs{Y_1 - Y_2}}{1 - \abs{Y_1}} \\
&  \lesssim_{\norm[\infty]{\w_0}, \Omega} \abs{\ln(1 - \abs{Y_1})}\abs{Y_1 - Y_2} \\
&  \lesssim_{\norm[\infty]{\w_0}, \Omega} \phi(\abs{Y_1 - Y_2})
\end{align*}

\smallskip
\textbf{Case 2}: $\abs{Y_1(y,t) - Y_2(y,t)} \geq 1 - \abs{Y_1(y,t)}$. 

This implies that $\phi(1 - \abs{Y_1(y,t)}) \leq \phi(\abs{Y_1(y,t) - Y_2(y,t)})$.  Now we simply use the bound on $F_r$ from \eqref{eq:LrLthetaest} and \lemref{lem:btil} part (4) to see that
\begin{align*}
& \abs{ \Real\cbrac{\btil_1(Y_1)\frac{\Ybar_1}{\abs{Y_1}} } \brac{\abs{\Psi_z(Y_1)}^{-2}\partial_r F(Y_1) - \abs{\Psi_z(Y_2)}^{-2}\partial_r F(Y_2)  }} \\ 
&  \lesssim_{\norm[\infty]{\w_0}, \Omega} \abs{\Real\cbrac{\btil_1(Y_1)\frac{\Ybar_1}{\abs{Y_1}} }} \\
&  \lesssim_{\norm[\infty]{\w_0}, \Omega} \phi(1 - \abs{Y_1}) \\
& \lesssim_{\norm[\infty]{\w_0}, \Omega} \phi(\abs{Y_1 - Y_2})
\end{align*}

Therefore this shows that 
\begin{align}\label{eq:IIestimated}
\rom{2} \lesssim_{\norm[\infty]{\w_0}, \Omega} \int_{\frac{1}{8} < \abs{y} < 1}  \phi(\abs{Y_1 - Y_2})\abs{\Psi_z(y)}^2 \diff y
\end{align}

\smallskip
\noindent$\mathbf{Controlling \enspace \rom{3}:}$ Using \lemref{lem:Fbilipschitz} and the lower bound $0< c_0 \leq \abs{\Psi_z(z)}$ we see that
\begin{align*}
\rom{3} & =  \int_\Dsp \abs{ \brac{(\btil_1 - \btil_2)(Y_2)\abs{\Psi_z(Y_2)}^{-2} }\cdot \grad F(Y_2) } \abs{\Psi_z(y)}^2 \diff y  \\
& \lesssim_{\Omega}  \int_\Dsp \abs{ \nobrac{(\btil_1 - \btil_2)(Y_2) } }   \abs{\Psi_z(y)}^2 \diff y
\end{align*}
Now as $X(\cdot, t), X^{-1}(\cdot, t) : \Omega \to \Omega$ preserve the Lebesgue measure $\diff x$, we see that the mappings $Y(\cdot, t), Y^{-1}(\cdot, t) : \Dsp \to \Dsp$ preserve the measure $\abs{\Psi_z(y)}^2 \diff y$. Hence using the change of variable $Y_2(y,t) \mapsto y$ we see that 
\begin{align*}
\rom{3} \lesssim_{\Omega}  \int_\Dsp \abs{ \nobrac{(\btil_1 - \btil_2)(y,t) } }  \diff y
\end{align*}
Now for $i = 1,2$ we have from \eqref{eq:btil} 
\begin{align*}
\btil_i(y,t) =  \brac*[\Big]{\frac{i}{2\pi}} \int_{\Dsp}\sqbrac{\frac{1}{\ybar - \sbar} - \frac{1}{\ybar - \frac{1}{s}}}\wtil_i(s,t)  \abs{\Psi_z (s)}^{2} \diff s
\end{align*}
From \lemref{lem:transport} we see that $\w_i(X_i(x,t), t) = \w_0(x)$ for $i = 1,2$. Hence we see that $\wtil_i(Y_i(y,t), t) = \wtil_0(y)$ for $i = 1,2$. Therefore applying the change of variable $s \mapsto Y_i(s,t)$ gives us
\begin{align*}
\btil_i(y,t) =  \brac*[\Big]{\frac{i}{2\pi}} \int_{\Dsp}\sqbrac{\frac{1}{\ybar - \Ybar_i(s,t)} - \frac{1}{\ybar - \frac{1}{Y_i(s,t)}}}\wtil_0(s)  \abs{\Psi_z (s)}^{2} \diff s
\end{align*}
This implies that 
\begin{align*}
 \abs{\btil_1(y,t) - \btil_2(y,t)} & \lesssim_{\norm[\infty]{\w_0}, \Omega} \int_\Dsp \frac{\abs{Y_1(s,t) - Y_2(s,t)}}{\abs{y - Y_1(s,t)}\abs{y - Y_2(s,t)} } \abs{\Psi_z(s)}^2 \diff s \\
& \quad + \int_\Dsp \frac{\abs{Y_1(s,t) - Y_2(s,t)}}{\abs{\ybar\, Y_1(s,t) - 1}\abs{\ybar\,  Y_2(s,t) - 1} } \abs{\Psi_z(s)}^2 \diff s
\end{align*}
Now observe that for $y ,z \in \Dsp$, we have $\abs{\ybar z - 1} \geq \abs{\ybar - \zbar}$. Therefore 
\begin{align*}
\abs{\btil_1(y,t) - \btil_2(y,t)}  \lesssim_{\norm[\infty]{\w_0}, \Omega} \int_\Dsp \frac{\abs{Y_1(s,t) - Y_2(s,t)}}{\abs{y - Y_1(s,t)}\abs{y - Y_2(s,t)} } \abs{\Psi_z(s)}^2 \diff s 
\end{align*}
This implies that 
\begin{align*}
\rom{3}  \lesssim_{\norm[\infty]{\w_0}, \Omega} \int_\Dsp \cbrac{\int_\Dsp \frac{\abs{Y_1(s,t) - Y_2(s,t)}}{\abs{y - Y_1(s,t)}\abs{y - Y_2(s,t)} } \abs{\Psi_z(s)}^2 \diff s}  \diff y
\end{align*}
Now using Fubini and \lemref{lem:phiabest} gives us
\begin{align}\label{eq:IIIestimated}
\rom{3}  \lesssim_{\norm[\infty]{\w_0}, \Omega} \int_\Dsp \phi(\abs{Y_1(s,t) - Y_2(s,t)})\abs{\Psi_z(s)}^2 \diff s
\end{align}

Therefore combining \eqref{eq:Iestimated}, \eqref{eq:IIestimated} and \eqref{eq:IIIestimated} gives us
\begin{align*}
E'(t)  \lesssim_{\norm[\infty]{\w_0}, \Omega} \int_\Dsp \phi(\abs{Y_1(y,t) - Y_2(y,t)})\abs{\Psi_z(y)}^2 \diff y
\end{align*}
Let $\tilde{c} = \tilde{c}(\Omega) > 0 $ be such that $\tilde{c} \abs{\Psi_z(y)}^2 \diff y$ is a probability measure on $\Dsp$. Now using \lemref{lem:Fbilipschitz} we see that 
\begin{align*}
E'(t)  \lesssim_{\norm[\infty]{\w_0}, \Omega} \int_\Dsp \phi(\abs{F(Y_1(y,t)) - F(Y_2(y,t))}) (\tilde{c}\abs{\Psi_z(y)}^2) \diff y
\end{align*} 
Now from \eqref{eq:FYoneminusFYtwoboundtime} and the fact that $\phi$ is concave on $[0, \frac{1}{10}]$ we obtain from Jensen's inequality
\begin{align*}
E'(t)  & \lesssim_{\norm[\infty]{\w_0}, \Omega} \phi\brac{\int_\Dsp \abs{F(Y_1(y,t)) - F(Y_2(y,t))} (\tilde{c}\abs{\Psi_z(y)}^2) \diff y} \\
& \lesssim_{\norm[\infty]{\w_0}, \Omega} \phi(E(t))
\end{align*}
Hence by Gr\"onwall's inequality we see that $E(t) = 0$ for all $t \in [0,T]$. From \lemref{lem:Fbilipschitz} this gives us $Y_1(y,t) = Y_2(y,t)$ for a.e. $(y,t) \in \Dsp \times [0,T]$. Therefore $(u_1, \w_1) = (u_2, \w_2)$ a.e.  on $\Omega \times [0,T]$. Now as $T = T(\norm[\infty]{\w_0}, \Omega)$ does not depend on the profile of the vorticity, we can repeat this argument for the time intervals $[T, 2T], [2T, 3T]$ and so on, to give us uniqueness on $[0,\infty)$. Hence proved. 
\end{proof}

\begin{proof}[Proof of \propref{prop:main}]
If $\Omega$ is a bounded simply connected $C^{1,\alpha}$ domain for some $0<\alpha < 1$, then  $\log(\Psi_z)$ extends continuously to $\Dspbar$ (See proof of Theorem 3.5 in \cite{Po92}. Here we fix the ambiguity in the definition of $\log$ by choosing a value of $\Imag(\log(\Psi_z(0)))$. The choice one makes is irrelevant). Hence there exists constants $c_0, c_1 > 0$ so that $c_0 \leq \abs{\Psi_z(s)} \leq c_1$ for all $z \in \Dspbar$ and therefore the first assumption of \assref{ass:main} is satisfied. Now by the Kellog-Warschawki  theorem (see Theorem 3.6 in \cite{Po92}) we see that $\Psi_z $ extends continuously to $\Dspbar$ and is $C^{\alpha}$ on $\Dspbar$. Hence $\abs{\Psi_z}^2$ is $C^{\alpha}$ on $\Dspbar$ and hence 
\begin{align*}
\sup_{r \in [0,1)}\norm[C^\alpha]{\abs{\Gcal}^2(r, \cdot)} \lesssim 1
\end{align*}
As the Hilbert transform $\Hil$ is bounded on $C^{\alpha}(\Sone)$, we see that 
\begin{align*}
\sup_{r \in [0,1)} \norm[\Linfty(\Sone)]{\Hil(\abs{\Gcal}^2(r, \cdot))} \lesssim \sup_{r \in [0,1)} \norm[C^{\alpha}(\Sone)]{\Hil(\abs{\Gcal}^2(r, \cdot))} \lesssim 1
\end{align*}
Therefore the third assumption of \assref{ass:main} is satisfied. Now let us prove the second assumption in all three cases:
\begin{enumerate}
\item Consider the $2\pi$ periodic function $Z : \Rsp \to \partial \Omega$ given by $\Z(\ap) = \Psi(e^{i\ap})$. From this we see that 
\begin{align}\label{eq:Zapdef}
\Z_\ap(\ap) = i\Psi_z(e^{i\ap})e^{i\ap}
\end{align} 
Now consider the given $2\pi$ periodic function $f:\Rsp \to \partial \Omega$ satisfying $\abs{f'(\ap)} \neq 0$ for all $\ap \in [0,2\pi]$ and that $f(\ap) \neq f(\bp) $ for $\ap, \bp \in [0,2\pi)$. We assume without loss of generality that the parametrization $f$ is counter clockwise. As $f'$ is continuous, we see that there exists $c>0$ such that 
\begin{align}\label{eq:faplowerupperbound}
0<  c \leq \abs{f_\ap(\ap)} \leq \frac{1}{c} < \infty \qq\tx{ for all } \ap \in \Rsp
\end{align}
Therefore we have a strictly increasing homeomorphism $h : \Rsp \to \Rsp$ such that $\Z(\ap) = f(h(\ap))$ for all $\ap \in \Rsp$, and $h(2n\pi) = 2n\pi$ for $n \in \Zsp$ and $h(\ap + 2\pi) = h(\ap) + 2\pi$. Therefore we see that
\begin{align}\label{eq:Zap}
\Z_\ap(\ap) = f'(h(\ap))h'(\ap)
\end{align}
Now let $l, m: \Dspbar \to \Rsp$ be continuous functions such that $\log(\Psi_z) = l + im$. By taking absolute values in \eqref{eq:Zap} we see that 
\begin{align}\label{eq:hderivbound}
h'(\ap) = \frac{e^{l(1, \ap)}}{\abs{f'(h(\ap))}}
\end{align}
Therefore we see that $h$ is bi-Lipschitz. Now as $f' \in \Hhalf(\Sone)$ and from \eqref{eq:faplowerupperbound} and \lemref{lem:Hhalfformula} we see that $\frac{f'}{\abs{f'}} \in \Hhalf(\Sone)$. As $h$ is bi-Lipschitz, from \lemref{lem:Hhalfformula} we now see that $\frac{f'}{\abs{f'}}\compose \h \in \Hhalf(\Sone)$. However from \eqref{eq:Zapdef}, \eqref{eq:Zap} we obtain
\begin{align*}
e^{im(1, \ap)} = -ie^{-i\ap}\frac{\Zap(\ap)}{\abs{\Zap(\ap)}}  = -ie^{-i\ap} \frac{f'(h(\ap))}{\abs{f'(h(\ap))}}
\end{align*}
Therefore $e^{im(1 ,\cdot)} \in \Hhalf(\Sone)$.  Now \lemref{lem:fefHhalf} below implies that $m(1, \cdot) \in \Hhalf(\Sone)$. From the fact that $\log(\Psi_z)$ is holomorphic, we see that we have $l(1, \cdot) - \Avg(l(1, \cdot)) = -\Hil(m(1, \cdot))$. Hence we see that $l(1, \cdot) \in \Hhalf(\Sone)$. Therefore again from \lemref{lem:fefHhalf} we see that $\Psi_z(1, \cdot) \in \Hhalf(\Sone)$. Therefore the second assumption of \assref{ass:main} is satisfied. 

\item From the Kellog-Warschawki theorem we see that there exists a $C>0$ so that $\abs{\Psi_z(z_1)  - \Psi_z(z_2)} \leq C\abs{z_1 - z_2}^\alpha$ for all $z_1, z_2 \in \Dspbar$. As $\alpha > \half$, we see from \lemref{lem:Hhalfformula} that $\Psi_z(1, \cdot) \in \Hhalf(\Sone)$. Therefore the second assumption of \assref{ass:main} is satisfied. 

\item For $0 < r < 1$, let $\Omega_r = \Psi(B_r(0))$. As $\Omega$ is convex this implies that $\Omega_r$ is a smooth convex domain for all $0 < r < 1$ (see Chapter 5, page 222 in \cite{Ne75}). Let $Z : (\frac{1}{2}, 1) \times \Rsp \to \Csp$ be given by $Z(r, \ap) = \Psi(re^{i\ap})$. Now again as in the first example, let $l, m: \Dspbar \to \Rsp$ be continuous functions such that $\log(\Psi_z) = l + im$. Therefore 
\begin{align*}
\Z_\ap(r, \ap) = ir\Psi_z(re^{i\ap})e^{i\ap} = ire^{l(r,\ap) + im(r,\ap)}e^{i\ap}
\end{align*}
As $\Omega_r$ is convex for $r \in (1/2, 1)$, this implies that $m(r,\ap) + \ap$ is an increasing function in $\ap$ for each fixed $r \in (1/2, 1)$. As $m(r, 0) = m(r, 2\pi)$ this implies that for $r \in (1/2, 1)$ we have
\begin{align*}
\int_0^{2\pi} \abs{m_\ap(r, \ap)} \diff \ap \leq 2\pi + \int_0^{2\pi} \abs{m_\ap(r, \ap) + 1} \diff \ap =  2\pi + \int_0^{2\pi} \brac{m_\ap(r, \ap) + 1} \diff \ap = 4\pi
\end{align*}
Now as $\Psi_z$ is $C^{\alpha}$ on $\Dspbar$, we see that $\log(\Psi_z)$ is $C^{\alpha}$ on $\Dspbar$ and hence 
\begin{align*}
\sup_{r \in (\frac{1}{2}, 1)} \norm[\Linfty(\Sone)]{\Hil (m(r, \cdot))} \lesssim \sup_{r \in (\frac{1}{2}, 1)} \norm[C^{\alpha}(\Sone)]{\Hil (m(r, \cdot))} \lesssim_\Omega 1
\end{align*}
Therefore we see that for all $r \in (1/2, 1) $ we have 
\begin{align*}
\norm[\Hhalf(\Sone)]{m(r, \cdot)}^2 \lesssim  \norm[\Linfty(\Sone)]{\Hil (m(r, \cdot))} \norm[\Lone(\Sone)]{m_\ap(r, \cdot)} \lesssim_\Omega 1
\end{align*}
As $l(r, \cdot) - \Avg(l(r, \cdot)) = -\Hil(m(r, \cdot))$, we see that $\sup_{r \in (\half, 1)}\norm[\Hhalf(\Sone)]{l(r, \cdot)} \lesssim_\Omega 1$. Now as $l + im$ is a bounded function on $\Dspbar$, we see from  \lemref{lem:fefHhalf} that 
\begin{align*}
\sup_{r \in (\frac{1}{2}, 1)} \norm[\Hhalf(\Sone)]{\Psi_z(r, \cdot)} \lesssim_\Omega 1
\end{align*}
Hence assumption two of \assref{ass:main} is satisfied. 

\end{enumerate}
\end{proof}

\section{Appendix}\label{sec:appendix}

Here we collect some basic estimates we use throughout the paper. 

\begin{lemma}\label{lem:phi}
Let $T,R,c>0$ and let $y:[0,T] \to \Rsp^{+}$ be such that $\abs{y(t)}\leq R$ for all $t\in [0,T]$ and satisfy 
\begin{align*}
\abs{\frac{\diff y}{\diff t}} \leq c\phi(y(t)) \qq\qq y(0) = y_0>0,
\end{align*}
where $\phi$ is given by \eqref{def:phi}. Then 
\begin{align*}
\sqbrac{\frac{y(0)}{10(R + 1)}}^{e^{ct}} \leq  y(t) \leq \nobrac{10(R+1)} \cbrac{y(0)}^{e^{-ct}} \qq \tx{ for all } t\in[0,T].
\end{align*}
\end{lemma}
\begin{proof}
From the proof of Lemma 5.1 in  \cite{AgNa22} we get that
\begin{align*}
\ln(y(t)) \leq e^{-ct}\ln(y(0)) + 1 + \ln(R + 1) \leq e^{-ct}\ln(y(0)) + \ln(10(R+1))
\end{align*}
Therefore from this we get the upper bound. To get the lower bound we follow the same proof there to get the inequality
\begin{align*}
\ln(y(t)) \geq e^{ct}\ln(y(0)) - e^{ct}\cbrac{1 + \ln(R + 1)} \geq e^{ct}\ln(y(0)) -e^{ct}\ln(10(R+1))
\end{align*}
From this we get the lower bound. 
\end{proof}

\begin{lemma}\label{lem:phiabest}
For $a, b \in \Dspbar$ we have
\begin{align*}
\int_\Dsp \frac{\abs{a - b}}{\abs{s - a} \abs{s - b}} \diff s \lesssim \phi(\abs{a - b})
\end{align*}
\end{lemma}
\begin{proof}
This is a standard lemma. See \cite{AgNa22} for a proof. 
\end{proof}

\begin{lemma}\label{lem:Hhalfformula}
Let $f:\Sone \to \Rsp$ be smooth. Then we have the formula
\begin{align}\label{eq:fabsfform}
(f\papabs f)(\alpha) = \frac{1}{8\pi} \int_0^{2\pi} \frac{(f(\ap) - f(\bp))^2}{\sin^2\brac{\frac{\ap - \bp}{2}}} \diff \bp + \frac{1}{2} \papabs(f^2)
\end{align}
In particular 
\begin{align}\label{eq:fhalf2normform}
\norm[\Hhalf]{f}^2 =  \frac{1}{8\pi} \int_0^{2\pi}\int_0^{2\pi} \frac{(f(\ap) - f(\bp))^2}{\sin^2\brac{\frac{\ap - \bp}{2}}} \diff \bp \diff \ap
\end{align}
\end{lemma}
\begin{proof}
As $\papabs = \Hil \pap$ we see that 
\begin{align*}
f \papabs f = f\Hil\pap f = \sqbrac{f,\Hil}\pap f + \Hil(f \pap f) = \sqbrac{f, \Hil} \pap f + \frac{1}{2}\papabs(f^2)
\end{align*}
Now we see that 
\begin{align*}
& (\sqbrac{f, \Hil} \pap f)(\ap) \\
& = \frac{1}{2\pi} \int_0^{2\pi} (f(\ap) - f(\bp))\cot\brac{\frac{\ap - \bp}{2}} f_\beta(\bp) \diff \beta \\
& = - \frac{1}{2\pi} \int_0^{2\pi} (f(\ap) - f(\bp))\cot\brac{\frac{\ap - \bp}{2}} f_\beta(\bp) \diff \beta + \frac{1}{4\pi} \int_0^{2\pi} \frac{(f(\ap) - f(\bp))^2}{\sin^2\brac{\frac{\ap - \bp}{2}}} \diff \bp \\
& = \frac{1}{8\pi} \int_0^{2\pi} \frac{(f(\ap) - f(\bp))^2}{\sin^2\brac{\frac{\ap - \bp}{2}}} \diff \bp
\end{align*}
Hence proved. 
\end{proof}

\begin{lemma}\label{lem:fefHhalf}
Let $f:\Sone \to \Csp$ be bounded. Then $\norm[\Hhalf(\Sone)]{e^f} \lesssim_{\norm[\infty]{f}} \norm[\Hhalf(\Sone)]{f}$. If $f$ is continuous and $e^f \in \Hhalf(\Sone)$, then $f \in \Hhalf(\Sone)$. 
\begin{proof}
Observe that for $z_1, z_2 \in \Csp$ with $\abs{z_1}, \abs{z_2} \leq M$ for some $M>0$ we have $\abs{e^{z_1} - e^{z_2}} \lesssim_M \abs{z_1 - z_2}$. 
From this and \lemref{lem:Hhalfformula} we easily get $\norm[\Hhalf(\Sone)]{e^f} \lesssim_{\norm[\infty]{f}} \norm[\Hhalf(\Sone)]{f}$. Now assume that $f$ is continuous on $\Sone$. This means that $f$ is uniformly continuous. Hence there exists $\delta >0$ such that if $\abs{e^{i\alpha} - e^{i\beta}} < \delta$ then $\abs{f(e^{i\alpha}) - f(e^{i\beta})} < \frac{1}{10}$. Now observe that for $z_1, z_2 \in \Csp$ with $\abs{z_1}, \abs{z_2} \leq M$ for some $M>0$ and $\abs{z_1 - z_2} < \frac{1}{10}$, we have $\abs{e^{z_1} - e^{z_2}} \approx_M \abs{z_1 - z_2}$. 
Therefore we have
\begingroup
\allowdisplaybreaks
\begin{align*}
1 + \norm[\Hhalf]{f}^2 & \approx 1 + \int_0^{2\pi} \int_0^{2\pi} \frac{\abs{f(e^{i\ap}) - f(e^{i\bp})}^2}{\sin^2\brac{\frac{\ap - \bp}{2}}} \diff \bp \diff \ap \\
& \approx_{\norm[\infty]{f}, \delta} 1 + \int_{\abs*{e^{i\ap} - e^{i\bp}  } < \delta } \frac{\abs{f(e^{i\ap}) - f(e^{i\bp})}^2}{\sin^2\brac{\frac{\ap - \bp}{2}}} \diff \bp \diff \ap \\
& \approx_{\norm[\infty]{f}, \delta} 1 + \int_{\abs*{e^{i\ap} - e^{i\bp}  } < \delta } \frac{\abs*[\big]{e^{f(e^{i\ap})} - e^{f(e^{i\bp})}}^2}{\sin^2\brac{\frac{\ap - \bp}{2}}} \diff \bp \diff \ap \\
& \approx_{\norm[\infty]{f}, \delta} 1 + \int_0^{2\pi} \int_0^{2\pi}\frac{\abs*[\big]{e^{f(e^{i\ap})} - e^{f(e^{i\bp})}}^2}{\sin^2\brac{\frac{\ap - \bp}{2}}} \diff \bp \diff \ap \\
& \approx_{\norm[\infty]{f}, \delta} 1 + \norm*[\big][\Hhalf]{e^f}^2
\end{align*}
\endgroup
Hence proved. 
\end{proof}
\end{lemma}

\begin{lemma}\label{lem:gradunotinLp}
Fix $0 < \alpha_0 < 1$ and let $\Omega$ be a bounded simply connected $C^{1, \alpha_0}$ domain with a Jordan boundary such that it is not a $C^{1, \gamma}$ domain for any $\alpha_0< \gamma < 1$. Let $\w \in \Linfty(\Omega)$ satisfy $\w(x) \geq 0$ for all $x \in \Omega$ and $\norm[\Linfty(\Omega)]{\w} > 0$. Then the function $u: \Omega \to \Csp$ given by $u = \grad^\perp \Delta^{-1} \w$ satisfies $u \notin C^\beta(\Omega)$ for all $\alpha_0 < \beta <1$ and $ \grad u  \notin L^p(\Omega) $ for all $p > \frac{2}{1 - \alpha_0}$. 
\end{lemma}
\begin{proof}
As $\Omega$ is a bounded simply connected domain whose boundary is a Jordan curve, we see that the Riemann map $\Psi:\Dsp \to \Omega$ extends continuously to $\Omegabar$ and with $\Psi$ mapping $\Sone$ to $\partial\Omega$ homeomorphically. As $\Omega$ is a $C^{1,\alpha_0}$ domain, we see that $\Psi_z$ also extends continuously to $\Dspbar$ and $1 \lesssim_\Omega \abs{\Psi_z(z)}\lesssim_\Omega 1$ for all $z \in \Dspbar$. Now let $\alpha_0 < \beta < 1$. As $\Omega$ is not a $C^{1,\beta}$ domain, this means that the mapping $\Psi\vert_{\Sone} : \Sone \to \partial \Omega$ is not $C^{1, \beta}$. This in particular implies that $\Psi_z \notin C^{\beta}(\Omega)$. Hence there exists $y_0 \in \Sone$ and sequences $(a_n)_{n=1}^\infty, (b_n)_{n=1}^\infty$ in $\Dsp$ with $a_n \neq b_n$ for all $n \geq 1$ and $a_n, b_n \to y_0$ such that
\begin{align}\label{eq:Psizanbncontra}
\lim_{n\to \infty} \frac{\abs{\Psi_z(a_n) - \Psi_z(b_n)}}{\abs{a_n - b_n}^\beta} = \infty
\end{align}
From the same way \eqref{eq:ufrombiotSavart} was derived, we see that $u:\Omega \to \Csp$ is given by $u(x) = b(x)\Phizbar(x)$, where $\Phi = \Psi^{-1}$ and $b$ is given by \eqref{eq:b}. Let $\util :\Dsp \to \Csp$ be defined as $\util(y) = u(\Psi(y))$. Therefore we see that 
\begin{align}\label{eq:utilbtilpsizbar}
\util(y) = \frac{\btil(y)}{\Psizbar(y)}
\end{align}
where $\btil$ is given by \eqref{eq:btil}.  From \lemref{lem:btil} part (1) we see that $u$ is bounded. Now assume by contradiction that $u \in C^{\beta}(\Omega)$ (recall that we assumed $\alpha_0 < \beta < 1$ earlier). 
As $\Psi$ is bi-Lipschitz, this implies that $\util \in C^\beta(\Dsp)$. Now from \eqref{eq:btilboundary} and the fact that $\w \geq 0$ and $\norm[\Linfty(\Omega)]{\w} > 0$, we see that $\abs*[\big]{\btil(y)} > 0$ for all $y \in \Sone$. As $\btil$ is log-Lipschitz on $\Dspbar$ from \lemref{lem:btil} part (2), this implies from \eqref{eq:utilbtilpsizbar} that there exists $N >1$ large enough so that for all $n \geq N$ we have 
\begin{align*}
\abs{\Psi_z(a_n) - \Psi_z(b_n)} \leq C\abs{a_n - b_n}^\beta
\end{align*}
for some constant $C>0$. This contradicts \eqref{eq:Psizanbncontra} and hence this proves that $u \notin C^\beta(\Omega)$. Now by Sobolev embedding we therefore also get that  $ \grad u  \notin L^p(\Omega) $ for all $p > \frac{2}{1 - \alpha_0}$. 

\end{proof}


\bibliographystyle{amsplain}
\bibliography{/Users/siddhant/Desktop/Writing/Main.bib}

\end{document}